\documentclass[onefignum,onetabnum]{siamart171218}

\usepackage{lipsum}
\usepackage{amsfonts}
\usepackage{graphicx}
\usepackage{epstopdf}
\usepackage{algorithmic}
\usepackage{caption}
\usepackage{subcaption}
\usepackage[nocompress]{cite}

\ifpdf
  \DeclareGraphicsExtensions{.eps,.pdf,.png,.jpg}
\else
  \DeclareGraphicsExtensions{.eps}
\fi

\newsiamremark{remark}{Remark}
\newsiamremark{example}{Example}
\newsiamremark{assumption}{Assumption}
\newsiamremark{hypothesis}{Hypothesis}
\crefname{hypothesis}{Hypothesis}{Hypotheses}
\newsiamthm{claim}{Claim}

\headers{Contraction Rates for discretized Kinetic Langevin dynamics}{}

\title{Contraction and Convergence Rates for discretized Kinetic Langevin dynamics}{}

\author{Benedict Leimkuhler\thanks{Department of Mathematics, University of Edinburgh.,
Edinburgh, EH9 3FD, UK. (\email{b.leimkuhler@ed.ac.uk}, \email{dpaulin@ed.ac.uk}, \email{p.a.whalley@sms.ed.ac.uk})}
\and Daniel Paulin\footnotemark[1]
\and Peter A. Whalley\footnotemark[1]}

\usepackage{amsopn}

\makeatletter
\newcommand*{\addFileDependency}[1]{
  \typeout{(#1)}
  \@addtofilelist{#1}
  \IfFileExists{#1}{}{\typeout{No file #1.}}
}
\makeatother

\ifpdf
\hypersetup{
  pdftitle={Contraction and Convergence Rates for Discretized Kinetic Langevin Dynamics},
  pdfauthor={}
}
\fi{}

\usepackage{todonotes}

\begin{document}

\maketitle

\begin{abstract}
We provide a framework to analyze the convergence of discretized kinetic Langevin dynamics for $M$-$\nabla$Lipschitz, {$m$-convex potentials}. Our approach gives convergence rates of $\mathcal{O}(m/M)$, with explicit stepsize restrictions, which are of the same order as the stability threshold for Gaussian targets and are valid for a large interval of the friction parameter. We apply this methodology to various integration schemes which are popular in the molecular dynamics and machine learning communities. Further, we introduce the property ``$\gamma$-limit convergent" (GLC) to characterize underdamped Langevin schemes that converge to overdamped dynamics in the high-friction limit and which have stepsize restrictions that are independent of the friction parameter; we show that this property is not generic by exhibiting methods from both the class and its complement.  {Finally, we provide asymptotic bias estimates for the BAOAB scheme, which remain accurate in the high-friction limit by comparison to a modified stochastic dynamics which preserves the invariant measure.}
\end{abstract}


\section{Introduction}
In this article, we study the following form of the Langevin dynamics equations (``kinetic Langevin dynamics") on $\mathbb{R}^{2d}$:
\begin{equation}\label{eq:underdamped_langevin}
\begin{split}
dX_{t} &= V_{t}dt,\\
dV_{t} &= -\nabla U(X_t) dt - \gamma V_{t}dt + \sqrt{2\gamma}dW_t,  
\end{split}
\end{equation}
where $U:\mathbb{R}^{d}\to \mathbb{R}$ is the potential energy, $\gamma > 0$ is a friction parameter,  and $W_{t}$ is $d$-dimensional standard Brownian motion.  Under mild conditions, it can be shown that the system \eqref{eq:underdamped_langevin} has an invariant measure $\pi$ with density proportional to $\exp(-U(X)-||V||^2/2)$ \cite{pavliotis2014stochastic}. Normally, Langevin dynamics is developed in the physical setting with additional parameters representing temperature and mass.  However, our primary aim in using (\ref{eq:underdamped_langevin}) is, ultimately, the computation of statistical averages involving only the position $X$, and in such situations, both parameters can be neglected without any loss of generality, or alternatively incorporated into our results through suitable rescalings of time and potential energy.  In this article, we focus on the properties of (\ref{eq:underdamped_langevin}) in relation to numerical discretization and variation of the friction coefficient.

Taking the limit as $\gamma \to \infty$ in (\ref{eq:underdamped_langevin}), and introducing a suitable time-rescaling ($t' =\gamma t$) results in the overdamped Langevin dynamics given by (see \cite{pavliotis2014stochastic}[Sec 6.5])
\begin{equation}\label{eq:overdamped}
dX_{t}  = -\nabla U(X_t) dt + \sqrt{2}dW_t.
\end{equation}
This equation has, again, a unique invariant measure with density proportional to $\exp(-U(x))$. Under the assumption of a Poincar\'e inequality, convergence rate guarantees can be established for the continuous dynamics \cite{bakry2014analysis}. In the case of kinetic dynamics, a more delicate argument is needed to establish exponential convergence, due to the hypoelliptic nature of the SDE (see \cite{cao2019explicit,villani2006hypocoercivity,dolbeault2015hypocoercivity,dolbeault2009hypocoercivity,bakry2006diffusions,baudoin2016wasserstein,baudoin2017bakry}). 

Langevin dynamics, in its kinetic and overdamped forms, is the basis of many widely used sampling algorithms in machine learning and statistics \cite{cheng2018underdamped,welling2011bayesian,vollmer2016exploration}. In sampling, Langevin dynamics is discretized and the individual timesteps generated by integration are viewed as approximate draws from the target distribution. The discretization with a step size $h>0$ will have an invariant measure $\pi_{h}$ on $\mathbb{R}^{2d}$. However, there is an inherent bias due to the finite difference numerical approximation $(\pi \neq \pi_{h})$. This bias is usually addressed by choosing a sufficiently small stepsize, or by adding bias correction by methods like Metropolis-Hastings adjustment. The choice of the discretization method has a significant effect on the quality of the samples and also on the computational cost of producing accurate samples, through stability properties, convergence rates, and asymptotic bias.

Overdamped Langevin dynamics has been heavily studied both in the continuous and the discretized settings, with popular integrators being the Euler-Maruyama and the limit method of the BAOAB scheme (see Section 4 of \cite{leimkuhler2013rational}). The kinetic Langevin system has been extensively studied in the continuous case, but there are still many open questions about the design of the numerical integrator. A metric that is typically used to quantify the performance of a sampling scheme is the number of steps required to reach a certain level of accuracy in Wasserstein distance. Non-asymptotic bounds in Wasserstein distance reflect computational complexity, convergence rate, and accuracy. Achieving such bounds relies on two steps: (1) determining explicit convergence rates of the process to its invariant measure and (2) proving non-asymptotic bias estimates for the invariant measure.

Discretized kinetic Langevin dynamics, without correction, is widely used, for example in molecular dynamics \cite{allen2017computer,frenkel2023understanding,leimkuhler2013rational}, where it is known to provide accurate approximations of physical quantities, even for complex, high dimensional potentials.  Metropolization is rarely used in this area, as it tends to slow down convergence by several orders of magnitude \cite{PhysRevE.75.056707}. Convergence properties of unadjusted Langevin dynamics were rigorously studied for the first time in the paper \cite{roberts1996exponential}. Further analysis of convergence of unadjusted methods is particularly important within the context of Multilevel Monte Carlo methods \cite{majka2020non}.

The analysis of the discretization bias requires different techniques and is available in the literature for many of the integrators we consider \cite{cheng2018underdamped,monmarche2021high,sanz2021wasserstein}. Where available, bias analysis can be combined with our contraction results to provide non-asymptotic guarantees. {In this paper, we also provide new asymptotic bias estimates for a scheme which remain finite in the high-friction limit, by comparing the scheme to an alternative scheme which exactly preserves the invariant measure. This scheme switches between exact Hamiltonian dynamics and Ornstein–Uhlenbeck process steps and we believe this technique can be extended to other integrators of this type.}

The approach that we use to obtain convergence rates is based on proving contraction for a  synchronous coupling, as in \cite{monmarche2021high,dalalyan2020sampling}. Proving contraction of a coupling has been a popular method for establishing convergence both in the continuous time setting and for the discretization for Langevin dynamics and Hamiltonian Monte Carlo (\cite{eberle2019couplings,bou2020coupling,deligiannidis2021randomized,bou2023mixing,bou2022couplings,riou2022metropolis,schuh2022global,bou2022unadjusted}), since a consequence of such a contraction is convergence in Wasserstein distance (viewed as the infimum over all possible couplings with respect to some norm). Synchronous coupling has been a popular means of achieving explicit convergence rates for discretizations due to its simplicity.

There have been other recent work aimed at providing convergence rates for kinetic Langevin dynamics under explicit restrictions on the parameters (\cite{cheng2018underdamped,dalalyan2020sampling,monmarche2021high,monmarche2022hmc}), but these guarantees are valid only with sharp restrictions on stepsize. There has also been the work of \cite{sanz2021wasserstein} which considers a slightly different version of the SDE (\ref{eq:underdamped_langevin}), where time is rescaled depending on the smallest and largest eigenvalues of the Hessian to optimize contraction rates and bias. We have included their results in Table \ref{Table:results} after converting them into our framework using \cite{dalalyan2020sampling}[Lemma 1]. The results of \cite{sanz2021wasserstein} rely on a stepsize restriction of $\mathcal{O}(1/\gamma)$, but their analysis does not provide the stepsize threshold \cite{sanz2021wasserstein}[Example 9], and the class of schemes considered is different, with only the stochastic Euler scheme in common. {Further, the techniques they use to quantify the asymptotic bias are not appropriate for many of the schemes considered in this article, as they are designed for high strong-order numerical integrators. For example, the BAOAB scheme is only strong order one but has asymptotic bias of order two, which cannot be estimated by their approach.}
Other works on contraction of kinetic Langevin and its discretization include \cite{foster2021shifted, dalalyan2022bounding,zhang2023improved}.

In the current article, we apply direct convergence analysis to various popular integration methods and provide a general framework for establishing convergence rates of kinetic Langevin dynamics with tight explicit stepsize restrictions of $\mathcal{O}\left(1/\gamma\right)$ or $\mathcal{O}(1/\sqrt{M})$ (depending on the scheme). As a consequence, we improve the contraction rates significantly for many of the available algorithms (see Table \ref{Table:results}). For a specific class of schemes, we establish explicit bounds on the convergence rate for stepsizes of $\mathcal{O}(1/\sqrt{M})$. In the limit of large friction, we distinguish two types of integrators -- those that converge to overdamped dynamics (``$\gamma$-limit-convergent") and those that do not.   We demonstrate with examples that this property is not universal: some seemingly reasonable methods have the property that the convergence rate falls to zero in the $\gamma\rightarrow \infty$ limit. This is verified numerically and analytically for an anisotropic Gaussian target. {Further, our novel asymptotic bias estimates for the BAOAB scheme demonstrate accuracy in the high-friction limit.}

The remainder of this article is structured as follows. We first introduce overdamped Langevin dynamics, the Euler-Maruyama (EM) and the high-friction limit of BAOAB (LM) and discuss their convergence guarantees. Next, we introduce kinetic Langevin dynamics and describe various popular discretizations, and give our results on convergence guarantees with mild stepsize assumptions. These schemes include first and second-order splittings and the stochastic Euler scheme (SES). Further, we compare the results of overdamped Langevin and kinetic Langevin and show how schemes like BAOAB and OBABO exhibit the positive qualities of both cases with the ``$\gamma$-limit convergent" property, whereas schemes like EM and SES do not perform well for a large range of $\gamma$. {Finally, we give asymptotic bias estimates for the BAOAB scheme which support the ``$\gamma$-limit convergent" property.}

\begin{table}
\begin{center}
\begin{tabular}{ |c|c|c|c|c } 
\hline
Algorithm & stepsize restriction & optimal one-step contraction rate \\
\hline
EM & $\mathcal{O}(1/\gamma)$ & $\mathcal{O}(m/M)$\\ 
BAO, OBA, AOB  &$\mathcal{O}(1/\sqrt{M})$   & $\mathcal{O}(m/M)$\\
OAB, ABO, BOA &$\mathcal{O}(1/\gamma)$  & $\mathcal{O}(m/M)$\\
BAOAB &$\mathcal{O}(1/\sqrt{M})$ & $\mathcal{O}(m/M)$ \\ 
OBABO &$\mathcal{O}(1/\sqrt{M})$ & $\mathcal{O}(m/M)$ \\ 
SES &$\mathcal{O}(1/\gamma)$ & $\mathcal{O}(m/M)$\\
\hline
\end{tabular}

\vspace{5mm}
\begin{tabular}{ |c|c|c| } 
\hline
Algorithm & previous stepsize restriction & previous explicit best rate\\
\hline
OBABO &$\mathcal{O}(m/\gamma^{3})$ & $\mathcal{O}(m^{2}/M^{2})$ \cite{monmarche2021high} \\ 
SES &$\mathcal{O}(1/\gamma)$ & $\mathcal{O}(m/M)$ \cite{sanz2021wasserstein}\\
\hline
\end{tabular}
\end{center}
\caption{The first table provides our stepsize restrictions and optimal contraction rates of the discretized kinetic Langevin dynamics. The second provides the previous best results. There are no previous results regarding the EM scheme, the first order splittings and BAOAB to the best of our knowledge.}
\label{Table:results}
\end{table}

\section{Assumptions and definitions}

\subsection{Assumptions on $U$} \label{sec:assumptions}
We will make the following assumptions on the potential $U:\mathbb{R}^{d}\to \mathbb{R}$.

\begin{assumption}[$M$-$\nabla$Lipschitz]\label{assum:G_Lipschitz} $U$ is twice continuously differentiable and there exists a $M > 0$ such that for all $X, Y \in \mathbb{R}^{d}$
\[
\left|\nabla U\left(X\right) - \nabla U\left(Y\right)\right| \leq M \left|X-Y\right|.
\]
\end{assumption}

\begin{assumption}[$m$-convexity]\label{assum:convex}
$U$ is continuously differentiable and there exists a $m > 0$ such that for all $X,Y \in \mathbb{R}^{d}$
\[
\left\langle \nabla U(X) - \nabla U(Y),X-Y \right\rangle \geq m \left|X-Y\right|^{2}.
\]
\end{assumption}

\begin{assumption}[$M_{1}$-Hessian Lipschitz]\label{assum:H_Lipschitz} $U$ is three times continuously differentiable and there exists a $M_{1} > 0$ such that for all $X, Y \in \mathbb{R}^{d}$
\[
\left|\nabla^{2} U\left(X\right) - \nabla^{2} U\left(Y\right)\right| \leq M_{1} \left|X-Y\right|.
\]
\end{assumption}

The first two assumptions are popular conditions used to obtain explicit convergence rates, see \cite{dalalyan2017theoretical,dalalyan2020sampling} for example. It is worth mentioning that these assumptions can also produce explicit convergence rates for gradient descent \cite{boyd2004convex}. {The final assumption is only used for proving higher order asymptotic bias estimates in Section \ref{sec:bias}}.

\subsection{Modified Euclidean Norms}\label{Sec:Quadratic_Norm}
For kinetic Langevin dynamics, it is not possible to prove convergence with respect to the standard Euclidean norm due to the fact that the generator is hypoelliptic. We therefore work with a modified Euclidean norm as in \cite{monmarche2021high}.
For $z = (x,v) \in \mathbb{R}^{2d}$ we introduce the weighted Euclidean norm
\[
\left|\left| z \right|\right|^{2}_{a,b} = \left|\left| x \right|\right|^{2} + 2b \left\langle x,v \right\rangle + a \left|\left| v \right|\right|^{2},
\]
for $a,b > 0$, which is equivalent to the Euclidean norm on $\mathbb{R}^{2d}$ as long as $b^{2}<a$. Under the condition $b^2<a/4$, we have
\[
\frac{1}{2}||z||^{2}_{a,0} \leq ||z||^{2}_{a,b} \leq \frac{3}{2}||z||^{2}_{a,0}.\]

\subsection{Wasserstein Distance} \label{sec:wasserstein_def}

We define $\mathcal{P}_{p}\left(\mathbb{R}^{2d}\right)$ to be the set of probability measures which have finite $p$-th moment, then for $p \in \left[0,\infty\right)$ we define the $p$-Wasserstein distance on this space. Let $\mu$ and $\nu$ be two probability measures. We define the $p$-Wasserstein distance between $\mu$ and $\nu$ with respect to the norm $||\cdot||_{a,b}$ (introduced in Section \ref{Sec:Quadratic_Norm}) to be 
\[\mathcal{W}_{p,a,b}\left(\nu,\mu\right) = \left( \inf_{\xi \in \Gamma\left( \nu, \mu \right)}\int_{\mathbb{R}^{2d}}||z_{1} - z_{2}||^{p}_{a,b}d\xi\left(z_{1},z_{2}\right)\right)^{1/p},\]
where $\Gamma\left(\mu,\nu\right)$ is the set of measures with marginals $\mu$ and $\nu$ (the set of all couplings between $\mu$ and $\nu$).

It is well known that the existence of couplings with a contractive property implies convergence in Wasserstein distance (which can be interpreted as the infimum over all such couplings). The simplest such coupling is to consider simulations with common noise, this is known as synchronous coupling, therefore if one can show contraction of two simulations which share noise increments with an explicit contraction rate. Then one has convergence in Wasserstein distance with the same rate. With all the constants and conditions derived for all the schemes for contraction, we have convergence in Wasserstein distance by the following proposition:
\begin{proposition}\label{prop:Wasserstein}
Assume a numerical scheme for kinetic Langevin dynamics with a $m$-strongly convex $M$-$\nabla$Lipschitz potential $U$ and transition kernel $P_{h}$. {Let $\left(x_{n},v_{n}\right)$ and $\left(\Tilde{x}_{n},\Tilde{v}_{n}\right)$ be} two synchronously coupled chains  of the numerical scheme {that} have the contraction property 
\begin{equation}\label{eq:contraction_inequality}
    ||(x_{n} - \Tilde{x}_{n},v_{n} - \Tilde{v}_{n})||^{2}_{a,b} \leq C(1 - c\left(h\right))^{n}||(x_{0} - \Tilde{x}_{0},v_{0} - \Tilde{v}_{0})||^{2}_{a,b},
\end{equation}
for $\gamma^{2} \geq C_{\gamma}M$ and $h \leq C_{h}\left(\gamma,\sqrt{M}\right)$ for some $a,b >0$ such that {$b^{2} < a/4$}. Then we have that for all $\gamma^{2} \geq C_{\gamma}M$, $h \leq C_{h}\left(\gamma,\sqrt{M}\right)$, $1 \leq p \leq \infty$ and all $\mu,\nu \in \mathcal{P}_{p}(\mathbb{R}^{2d})$, and  all $n \in \mathbb{N}$,
\[
\mathcal{W}^{2}_{p}\left(\nu P^n_{h} ,\mu P^n_{h} \right) \leq 3C\max{\left\{a,\frac{1}{a}\right\}}\left(1 - c\left(h\right)\right)^{n} \mathcal{W}^{2}_{p}\left(\nu,\mu\right).
\]
Further to this, $P_{h}$ has a unique invariant measure which depends on the stepsize, $\pi_{h}$, where $\pi_{h} \in \mathcal{P}_{p}(\mathbb{R}^{2d})$ for all $1 \leq p \leq \infty$.
\end{proposition}
\begin{proof}
 The proof is given in \cite{monmarche2021high}[Corollary 20], which relies on \cite{villani2009optimal}[Corollary 5.22, Theorem 6.18].
\end{proof}
The focus of this article is to prove contractions of the form (\ref{eq:contraction_inequality}), and hence to achieve Wasserstein convergence rates by Proposition \ref{prop:Wasserstein}. With convergence to the invariant measure of the discretizations of kinetic Langevin dynamics considered here it will be possible to combine our results with estimates of the bias of each scheme as in \cite{dalalyan2020sampling}, \cite{monmarche2021high}, \cite{sanz2021wasserstein} and \cite{cheng2018underdamped} to obtain non-asymptotic estimates.

\section{Overdamped Langevin discretizations and contraction}

We first consider two discretizations of the SDE (\ref{eq:overdamped}), namely the Euler-Maruyama discretization and the high-friction limit of the popular kinetic Langevin dynamics scheme BAOAB \cite{leimkuhler2013rational}. The simplest discretization of overdamped Langevin dynamics is using the Euler-Maruyama (EM) method which is defined by the update rule
\begin{equation}
 x_{n+1} = x_{n} - h\nabla U\left(x_{n}\right) + \sqrt{2h}\xi_{n+1},   
\end{equation}
where $(\xi_{n})_{n \in \mathbb{N}}$ are independent $d$-dimensional standard Gaussian random variables, that is $\xi_{n} \sim \mathcal{N}(0,I_{d})$ for all $n \in \mathbb{N}$, where $I_{d}$ is the $d$-dimensional identity matrix.
This scheme is combined with Metropolization in the popular MALA algorithm.

An alternative method is the BAOAB limit method of Leimkuhler and Matthews (LM)(\cite{leimkuhler2013rational}, \cite{leimkuhler2014long}) which is defined by the update rule
\begin{equation}\label{eq:LM}
x_{n+1} = x_{n} - h\nabla U\left(x_{n}\right) + \sqrt{2h}\frac{\xi_{n+1} + \xi_{n}}{2}.
\end{equation}
The advantage of this method is that it gains a weak order of accuracy asymptotically.

\subsection{Convergence guarantees} \label{sec:conv_overdamped}
The convergence guarantees of overdamped Langevin dynamics and its discretizations have been extensively studied under the assumptions presented (see \cite{dalalyan2017theoretical,durmus2017nonasymptotic,cheng2018convergence,dalalyan2017further,durmus2019high, durmus2019analysis, dwivedi2018log}). We use synchronous coupling as a proof strategy to obtain convergence rates as in \cite{dalalyan2017theoretical}. We first consider two chains $x_{n}$ and $y_{n}$ with shared noise such that
\begin{align*}
    x_{n+1} = x_{n} - h\nabla U(x_{n}) + \sqrt{2h}\xi_{n+1}, \quad y_{n+1} = y_{n} - h\nabla U(y_{n}) + \sqrt{2h}\xi_{n+1}.
\end{align*}
Then we have that 
\begin{align*}
    &||x_{n+1} - y_{n+1}||^2 = ||x_{n} - y_{n} + h(-\nabla U(x_{n}) - (-\nabla U(y_{n}))||^{2} \\
    &= ||x_{n} - y_{n}||^{2} - 2h \langle \nabla U(x_{n}) - \nabla U(y_{n}) , x_{n} - y_{n}\rangle + h^{2}||\nabla U(x_{n}) - \nabla U(y_{n})||^{2}\\
    &= ||x_{n} - y_{n}||^{2} - 2h \langle x_{n} - y_{n}, Q(x_{n} - y_{n})\rangle + h^{2}\langle x_{n} - y_{n}, Q^{2} (x_{n} - y_{n}) \rangle,
\end{align*}
where $Q = \int^{1}_{t = 0}\nabla^{2}U(x_{n} + t(y_{n} - x_{n}))dt$. $Q$ has eigenvalues which are bounded between $m$ and $M$, so $Q^{2} \preceq MQ$, and hence 
\[
h^{2}\langle x_{n} - y_{n}, Q^{2} (x_{n} - y_{n}) \rangle \leq h^{2}M \langle x_{n} - y_{n}, Q(x_{n} - y_{n}) \rangle.
\]
Therefore 
$
 ||x_{n+1} - y_{n+1}||^2 \leq ||x_{n} - y_{n}||^{2}(1 - hm(2-hM)),
$
assuming that $h \leq \frac{2}{M}$ we have contraction and 
\begin{equation}\label{eq:overdamped_contraction}
||x_{n} - y_{n}|| \leq (1 - hm(2 - hM))^{n/2}||x_{0} - y_{0}||.
\end{equation}
A consequence of this contraction result is that we have convergence in Wasserstein distance to the invariant measure with rate $hm\left(2 - hM\right)$, under the imposed assumptions on $h$ (as discussed in Section \ref{sec:wasserstein_def})\cite{monmarche2021high,villani2009optimal}. 

{
This argument is similar to the LM discretization \eqref{eq:LM} of overdamped Langevin dynamics. 
Note that $(x_n)_{n\ge 0}$ by itself does not define a Markov chain for this discretization, but by extending the state space to include the noise, $(x_n,\xi_n)_{n\ge 0}$ does. 
For two initial points $(x_0,\xi_0)$ and $(y_0,\xi_0')$, by using shared noise $\xi_n=\xi_n'$ for $n\ge 1$, it follows by the same argument that
\begin{equation}\label{eq:overdamped_contraction2}
||x_{n} - y_{n}|| \leq (1 - hm(2 - hM))^{n/2}||x_{1} - y_{1}||,
\end{equation}
for $n \geq 2$, since two processes have shared noise after $n = 0$.
}

The stepsize assumption for convergence of overdamped Langevin dynamics in this setting is weak and is the same assumption as is needed to guarantee convergence of gradient descent in optimization \cite{boyd2004convex}. 

\section{Kinetic Langevin Dynamics}
We now consider several discretization methods for the SDE (\ref{eq:underdamped_langevin}). {The aim is to} construct an alternative Euclidean norm in which we can prove contraction (it is not possible to prove contraction in the standard Euclidean norm). Essentially, we convert the problem of proving contraction to the problem of showing that certain matrices are positive definite. {First, in Section \ref{sec:discretisations} we introduce the discretization methods we consider; then in Section \ref{sec:strategy} we provide a framework for proving contraction for these discretizations, and in Section \ref{sec:results} we detail the contraction results for each of the schemes.}
\subsection{Discretization schemes}\label{sec:discretisations}

{We will consider some popular numerical discretization methods for kinetic Langevin dynamics, arising from molecular dynamics and machine learning.}

\subsubsection{The Euler-Maruyama method}

{First, we consider the simplest discretization, the Euler-Maruyama method. For the initial condition $(x_0,v_0) \in \mathbb{R}^{2d}$, the iterates $(x_n,v_n,\xi_n)$ for $n \in \mathbb{N}$ are defined by:
\begin{eqnarray}
x_{n+1} & = & x_n + hv_n,\\
v_{n+1} & = & v_n - h\nabla U(x_n) - h\gamma v_n + \sqrt{2\gamma h}\xi_{n+1},
\end{eqnarray}}
{where $(\xi_n)_{n\in\mathbb{N}}$ are independent ${\cal N}(0,I_{d})$ draws.}
\subsubsection{Splitting Methods} \label{sec:splittings}
More advanced integrators than Euler-Maruyama can be constructed based on splitting. These rely on an additive decomposition of the SDE into various terms which can be easily (often exactly) integrated.  A useful class of schemes relies on the exact integration of linear positional drift,  impulse due to the force and a dissipative-stochastic term corresponding to an Ornstein-Uhlenbeck equation \cite{PhysRevE.75.056707}. 
{
The solution maps corresponding to these parts may be denoted by $\mathcal{B}$, $\mathcal{A}$,  and $\mathcal{O}$ with update rules given by
\begin{equation}\label{eq:BAO}
\begin{split}
    &\mathcal{B}: \left(x, v\right) \to \left(x, v - h\nabla U(x)\right), \\
    &\mathcal{A}: \left(x,v\right)  \to \left(x + hv,v\right),\\
    &\mathcal{O}: \left(x,v\right) \to \left(x,\eta v + \sqrt{1 - \eta^{2}}\xi\right),
\end{split}  
\end{equation}
where $\xi \sim \mathcal{N}(0,I_{d})$ and
\[
\eta := \exp{\left(-\gamma h \right)}.
\] 
The infinitesimal generator of the SDE dynamics \eqref{eq:underdamped_langevin} can be split as 
 $\mathcal{L} = \mathcal{L}_{\mathcal{A}} + \mathcal{L}_{\mathcal{B}} + \gamma\mathcal{L}_{\mathcal{O}}$, where
 \begin{equation}
 \label{eq:split_gen}
 \mathcal{L}_{\mathcal{A}} = \left\langle v, \nabla_{x}\right\rangle, \qquad \mathcal{L}_{\mathcal{B}} = -\left\langle\nabla U\left(x\right), \nabla_{v}\right\rangle, \qquad \mathcal{L}_{\mathcal{O}} = -\left \langle v, \nabla_{v} \right\rangle + \Delta_{v},
\end{equation}
where $\mathcal{L}_{\mathcal{A}}$ and $\mathcal{L}_{\mathcal{B}}$
are the deterministic dynamics related to the both $\mathcal{A}$ and $\mathcal{B}$, while the dynamics of $\gamma\mathcal{L}_{\mathcal{O}}$
corresponds to the dynamics of an Ornstein-Uhlenbeck process.
The idea of splitting schemes is that we compose the maps $\mathcal{A}$, $\mathcal{B}$  and $\mathcal{O}$ as functions
in some way.    Each of the deterministic terms can be solved analytically, whereas 
$\exp(h \mathcal{L}_{\mathcal{O}})$
is realized by a weakly exact process. 
There are several possible options for the ordering of the different parts which are denoted by different strings.  For example, the scheme ABO would apply, in sequence, the A, B and O propagators. Other alternative options include BAO, BOA, AOB, OAB, and OBA.  Such schemes are called first order because they have a weak order of one.}

Alternatively, we can go beyond first-order methods to attain higher weak orders than one by using symmetric Strang splittings.
These are defined by palindromic letter sequences such as ABOBA, OBABO, etc. The interpretation of a string such as ABOBA is the following: we apply A for half a timestep (drift), then B for half a timestep (kick), generate a stochastic path corresponding to the Ornstein-Uhlenbeck equation in momentum, then follow with a half-step kick and finally a half-step drift. Such symmetric schemes can typically be applied using only one new force evaluation at each timestep (with the second force evaluation re-used at the start of the following step).  Despite this, the symmetry implies that they have second (weak) order of accuracy (see \cite{leimkuhler2013rational,leimkuhler2016computation}), meaning that the bias in long run averages is $O(h^2)$ for stepsize $h$ as $h\to 0$. Thus they are efficient in providing high accuracy at little additional cost compared to first-order methods.  Moreover, as shown in \cite{leimkuhler2013jcp}, a particular choice of splitting, namely the BAOAB method, has no bias at all for Gaussian targets.

{We will use the notation $\Psi_{\rm{BAOAB}}(x,v,h)$ throughout to denote the one-step map of the BAOAB discretization with initial conditions $(x,v) \in \mathbb{R}^{2d}$ and stepsize $h >0$, and similarly for other discretizations.}

\subsubsection{The stochastic exponential Euler method}

{See \cite{durmus2021uniform} for an introduction to the stochastic exponential Euler scheme and a derivation. This scheme is based on keeping the force constant and analytically integrating the whole process over a time interval. This scheme is the one considered in \cite{cheng2018underdamped,dalalyan2020sampling} and has gained a lot of attention in the machine learning community and we can apply our methods to this scheme. Similar schemes have also been considered in \cite{chandrasekhar1943stochastic,ermak1980numerical,skeel2002impulse} and it has been analyzed in 
\cite{durmus2021uniform,shi2012convergence}. In the notation we have used, it is defined by the updates
\begin{equation}\label{eq:SES}
    \begin{split}
    x_{n+1} &= x_{n} + \frac{1-\eta}{\gamma}v_{n} - \frac{\gamma h + \eta -1}{\gamma^{2}}\nabla U\left(x_{n}\right) + \zeta_{n+1},\\
    v_{n+1} &= \eta v_{n} - \frac{1 - \eta}{\gamma} \nabla U\left(x_{n}\right) + \omega_{n+1},
    \end{split}
\end{equation}
where
\begin{equation} \label{eq:SES_noise}
    \zeta_{n+1} = \sqrt{2\gamma}\int^{h}_{0}e^{-\gamma\left( h - s\right)}dW_{h\gamma + s}, \qquad \omega_{n+1} = \sqrt{2\gamma} \int^{h}_{0} \frac{1 - e^{-\gamma\left( h - s\right)}}{\gamma}dW_{h\gamma + s}.
\end{equation}
$\left(\zeta_{n},\omega_{n}\right)_{n \in \mathbb{N}}$ are Gaussian random vectors with covariances matrices which are given in \cite{durmus2021uniform}. We can couple two trajectories which have common noise $\left(\zeta_{n},\omega_{n}\right)_{n \in \mathbb{N}}$ to obtain contraction rates by the previously introduced methods.}

\subsection{Proof Strategy}\label{sec:strategy}
{To prove contraction and Wasserstein convergence of the kinetic Langevin integrators} we will consider a modified Euclidean norm as defined in Section \ref{Sec:Quadratic_Norm} for some choice of $a$ and $b$. We aim to construct an equivalent Euclidean norm such that contraction occurs for two Markov chains simulated by the same discretization scheme $z_{n} = (x_{n},v_{n}) \in \mathbb{R}^{2d}$ and $\Tilde{z}_{n} = (\Tilde{x}_{n},\Tilde{v}_{n}) \in \mathbb{R}^{2d}$ that are synchronously coupled. That is, for some choice of $a$ and $b$ such that $a,b >0$ and {$b^{2} < a/4$}
\begin{equation}\label{eq:cont_1}
  ||\Tilde{z}_{n+1} - z_{n+1}||^{2}_{a,b} < \left(1 - c\left(h\right)\right)||\Tilde{z}_{n} - z_{n}||^{2}_{a,b},  
\end{equation}
where $a$ and $b$ are chosen to provide reasonable explicit assumptions on the stepsize $h$ and friction parameter $\gamma$. Our initial choices of $a$ and $b$ for simple schemes are motivated by \cite{monmarche2021high}, and are derived by considering contraction of the continuous dynamics. Let $\overline{z}_{j} = \Tilde{z}_{j} - z_{j}$ for $j \in \mathbb{N}$, then (\ref{eq:cont_1}) is equivalent to showing that 
\begin{equation}\label{eq:contraction_matrix_form}
 \overline{z}^{T}_{n}\left(\left(1 - c\left(h\right)\right){G}- P^{T}{G}P\right )\overline{z}_{n} > 0,    \quad \textnormal{where} \quad {G} = \begin{pmatrix}
    I_{d} & bI_{d} \\
    bI_{d} & aI_{d}
\end{pmatrix},
\end{equation}
and $\overline{z}_{n+1} = P\overline{z}_{n}$ ($P$ depends on $z_{n}$ and $\Tilde{z}_{n}$, but we omit this in the notation).
\begin{example}
As an example, we have for the Euler-Maruyama method the update rule for $\overline{z}_{n}$
\begin{align*}
    \overline{x}_{n+1} = \overline{x}_{n} + h \overline{v}_{n}, \qquad  \overline{v}_{n+1} = \overline{v}_{n} - \gamma h \overline{v}_{n} -hQ\overline{x}_{n},
\end{align*}
where by the mean value theorem we can define $Q = \int^{1}_{t = 0}\nabla^{2}U(\Tilde{x}_{n} + t(x_{n} - \Tilde{x}_{n}))dt$, then $\nabla U(\Tilde{x}_{n}) - \nabla U(x_{n}) = Q\overline{x}$. One can show that in the notation of equation \eqref{eq:contraction_matrix_form} we have 
\begin{equation}\label{eq:P_matrix}
    P = \begin{pmatrix} I_{d} & hI_{d}\\
 -hQ & \left(1- \gamma h\right)I_{d} 
\end{pmatrix}.
\end{equation}  
\end{example}
Proving contraction for a general scheme is equivalent to showing that the matrix $\mathcal{H}:=  \left(1 - c(h)\right){G} - P^{T}{G}P \succ 0$ is positive definite. The matrix $\mathcal{H}$ is symmetric and hence of the form
\begin{equation}\label{eq:contraction_matrix}
\mathcal{H} = \begin{pmatrix}
    A & B \\
    B & C
\end{pmatrix},    
\end{equation}
we can show that $\mathcal{H}$ is positive definite by applying the following Proposition \ref{Prop:PD}.
\begin{proposition} \label{Prop:PD}
Let $\mathcal{H}$ be a symmetric matrix of the form (\ref{eq:contraction_matrix}), then $\mathcal{H}$ is positive definite if and only if $A \succ 0$ and $C - BA^{-1}B \succ 0$. Further if $A$, $B$ and $C$ commute then $\mathcal{H}$ is positive definite if and only if $A\succ 0$ and $AC - B^{2} \succ 0$.
\end{proposition}
\begin{proof}
The proof of the first result is given in \cite{horn2005basic}. To establish the second statement, observe from \cite{horn2012matrix}  that if two matrices are positive definite and they commute then the product is positive definite. Also if $A \succ 0$ then $A^{-1} \succ 0$ (as $A$ is symmetric positive definite). Further $A, B$ and $C$ commute and hence $B$, $C$ and $A^{-1}$ commute. Therefore by applying the first result, we have that $A \succ 0$ and
\[A^{-1}\left(AC - B^{2}\right) = C - BA^{-1}B \succ 0,\]
hence $\mathcal{H}$ is positive definite. If $\mathcal{H}$ is positive definite then $A \succ 0$ and  $C - BA^{-1}B \succ 0$ by the first result. Thus as $A$, $B$ and $C$ commute we have $AC - B^2 \succ 0$.
\end{proof}

\begin{remark}
An equivalent condition for a symmetric matrix $\mathcal{H}$ of the form (\ref{eq:contraction_matrix}) to be positive definite is $C\succ 0$ and $AC - B^{2} \succ 0$ when $A$, $B$ and $C$ commute. One could equivalently prove that $C \succ 0$ instead of $A \succ 0$ if it is more convenient.
\end{remark}

Our general approach to prove contraction of kinetic Langevin dynamics schemes is to prove the conditions of Proposition \ref{Prop:PD} are satisfied to establish contraction. We will use the notation laid out in this section in the proofs given in the appendix.
\subsection{{Convergence results}} \label{sec:results}

{We now detail contraction results of all the schemes introduced in Section \ref{sec:discretisations}. These results use the proof strategy of Section \ref{sec:strategy}.}

\subsubsection{Euler-Maruyama discretization}
We define the EM chain with initial condition $(x_0,v_0)$ by $(x_n,v_n,\xi_n)$ where the $(\xi_n)_{n\in\mathbb{N}}$ are independent ${\cal N}(0,1)$ draws and $(x_{n},v_n)$ are updated according to:
\begin{eqnarray*}
x_{n+1} & = & x_n + hv_n,\\
v_{n+1} & = & v_n - h\nabla U(x_n) - h\gamma v_n + \sqrt{2\gamma h}\xi_{n+1}.
\end{eqnarray*}

\begin{theorem} \label{Theorem:EM} Assume $U$ is an $m$-strongly convex and $M$-$\nabla$Lipschitz potential. When $\gamma^{2} \geq 4M$ and $h < \frac{1}{2\gamma}$, we have that, for all initial conditions $\left(x_{0},v_{0}\right) \in \mathbb{R}^{2d}$ and $\left(\Tilde{x}_{0},\Tilde{v}_{0}\right)\in \mathbb{R}^{2d}$, and for any sequence of standard normal random variables $\left(\xi_{n}\right)_{n \in \mathbb{N}}$, the corresponding EM chains $\left(x_{n},v_{n},\xi_{n}\right)_{n \in \mathbb{N}}$ and $\left(\Tilde{x}_{n},\Tilde{v}_{n},\xi_{n}\right)_{n \in \mathbb{N}}$ with initial conditions $\left(x_{0},v_{0}\right) \in \mathbb{R}^{2d}$ and $\left(\Tilde{x}_{0},\Tilde{v}_{0}\right)\in \mathbb{R}^{2d}$, respectively, satisfy
\[||(x_{n} - \Tilde{x}_{n},v_{n} - \Tilde{v}_{n})||_{a,b} \leq \left(1 - c\left(h\right)\right)^{\frac{n}{2}}||(x_{0} - \Tilde{x}_{0},v_{0} - \Tilde{v}_{0})||_{a,b},
    \]
where $a = \frac{1}{M}$, $b = \frac{1}{\gamma}$ and
$
c\left(h\right) = \frac{mh}{2\gamma}.
$
\end{theorem}

 \begin{example}
\textit{An example to illustrate the tightness of the restrictions on the stepsize $h$ and the restriction on the friction parameter $\gamma$.} We consider the anisotropic Gaussian distribution on $\mathbb{R}^{2}$ with potential $U: \mathbb{R}^{2} \mapsto \mathbb{R}$ given by $U(x,y) = \frac{1}{2}mx^{2} + \frac{1}{2}My^{2}$. This potential satisfies Assumptions \ref{assum:G_Lipschitz} and \ref{assum:convex} with constants $M$ and $m$ respectively. By computing the eigenvalues of the transition matrix $P$ (for contraction) we can see for what values of $h$ contraction occurs.  {For a Gaussian target, stability and asymptotic convergence speed are completely determined by the
eigenvalues of $P$ (which is constant)}. For EM we have that
 \[
P = \begin{pmatrix} I_{2} & hI_{2}\\
-hQ & \left(1- \gamma h\right)I_{2} 
\end{pmatrix}\text{, where }Q = \begin{pmatrix}
    m & 0 \\
    0 & M \end{pmatrix},
\]
with eigenvalues 
$\frac{1}{2}\left(2 - \gamma h \pm h \sqrt{\gamma^{2} - 4\lambda}\right)$,
for $\lambda = m,M$. For stability and contraction, we require that 
{
\begin{equation}\label{eq:condEM}
\lambda_{\max}:=\max_{\lambda\in \{m,M\}}\left|\frac{1}{2}\left(2 - \gamma h \pm h \sqrt{\gamma^{2} - 4\lambda} \right)\right| < 1. 
\end{equation}
By Gelfand's formula, the asymptotic contraction rate exactly equals $1-\lambda_{\max}$.
For $\gamma \geq 2\sqrt{M}$, all 4 eigenvalues are real, and this condition requires that $h < 4/(\gamma + \sqrt{\gamma^{2}-4\lambda}) \approx 2/\gamma$. Up to a constant, this is consistent with the stepsize restriction in our contraction rate results. $\lambda_{\max}$ in this range of $\gamma$ equals $1-\frac{1}{2} h\gamma \left(1-\sqrt{1-\frac{4m}{\gamma^2}}\right)$, so the best possible contraction rate is $O(m/M)$ (for the choice $\gamma=2\sqrt{M}$), which is also consistent with our results.\\ 
For $\gamma\in [2\sqrt{m},2\sqrt{M})$, the absolute value of the eigenvalues are
$
\sqrt{1-\gamma h+Mh^2}, 
$
and 
$
\frac{1}{2}\left|2-\gamma h \pm h\sqrt{\gamma^2-4m}\right|,
$
where we need all of these to be less than 1.
The first condition $1-\gamma h+Mh^2\le 1$ requires that $h\le \frac{\gamma}{M}$. The second condition $\frac{1}{2}\left|2-\gamma h-h\sqrt{\gamma^2-4m}\right|\le 1$ requires that $h\le \frac{4}{\gamma+\sqrt{\gamma^2-4m}}$. Using our assumption that $\gamma\in [2\sqrt{m},2\sqrt{M})$, $\gamma^2\le 4M$, so $\frac{4}{\gamma+\sqrt{\gamma^2-4m}}\ge \frac{2\gamma}{\gamma^2}\ge \frac{\gamma}{2M}$, and the second condition holds whenever $h\le \frac{\gamma}{2M}$. The third condition is satisfied whenever the second holds. Hence the stepsize restriction in this regime is $\frac{\gamma}{2M}$. One can show that the best possible convergence rate is still $O(m/M)$.
\\
Finally, when $\gamma<2\sqrt{m}$, \eqref{eq:condEM} becomes equivalent to $\sqrt{1-\gamma h+ Mh^2}<1$ and $\sqrt{1-\gamma h+mh^2}<1$, which results in the stepsize restriction $h\le \frac{\gamma}{M}$. The best possible convergence rate in this regime is still $O(m/M)$.
}
 \end{example}


\subsection{First order splittings}

We will now consider contraction for all first-order splitting methods (permutations of the $\mathcal{B}$, $\mathcal{A}$ and $\mathcal{O}$ pieces), which are schemes with weak order $1$. We first consider BAO, where we define a BAO chain with initial condition $\left(x_{0},v_{0}\right) \in \mathbb{R}^{2d}$ by $\left(x_{n},v_{n},\xi_{n}\right)_{n \in \mathbb{N}}$, using the update $\mathcal{BAO}$ (\ref{eq:BAO}) and $\left(\xi_{n}\right)_{n \in \mathbb{N}}$ are vectors of standard normal random variables.
 
\begin{theorem}[BAO] \label{Theorem:BAO} Assume $U$ is an $m$-strongly convex and $M$-$\nabla$Lipschitz potential. When $h < \frac{1 - \eta}{\sqrt{6M}}$, we have that for all initial conditions $\left(x_{0},v_{0}\right) \in \mathbb{R}^{2d}$ and $\left(\Tilde{x}_{0},\Tilde{v}_{0}\right)\in \mathbb{R}^{2d}$, and for any sequence of standard normal random variables $\left(\xi_{n}\right)_{n \in \mathbb{N}}$ the BAO chains $\left(x_{n},v_{n},\xi_{n}\right)_{n \in \mathbb{N}}$ and $\left(\Tilde{x}_{n},\Tilde{v}_{n},\xi_{n}\right)_{n \in \mathbb{N}}$ with initial conditions $\left(x_{0},v_{0}\right) \in \mathbb{R}^{2d}$ and $\left(\Tilde{x}_{0},\Tilde{v}_{0}\right)\in \mathbb{R}^{2d}$, respectively, satisfy
\[||(x_{n} - \Tilde{x}_{n},v_{n} - \Tilde{v}_{n})||_{a,b} \leq \left(1 - c\left(h\right)\right)^{\frac{n}{2}}||(x_{0} - \Tilde{x}_{0},v_{0} - \Tilde{v}_{0})||_{a,b},
    \]
where $a = \frac{1}{M}$ and $b = \frac{h}{1-\eta}$ and 
$
c\left(h\right) = \frac{h^{2}m}{4\left(1 - \eta\right)}.
$
\end{theorem}
\begin{remark}
The modified Euclidean norm has now been chosen to be stepsize-dependent and is needed to eliminate the corresponding dependency of the stepsize on the strong convexity constant $m$. We note that simply choosing {$a = 1/M$ and} $b = 1/\gamma$ does not result in a norm which guarantees a stepsize restriction which is independent of $m$, as is clear from the motivation of the construction of our choice of $b$. {This is due to the fact that one of the eigenvalues of $AC-B^{2}$ for BAO will be of the form $\mathcal{O}(m) - (b(1-\eta) - h)^{2} > 0$. Therefore if one fixes $h$ one can make $m$ sufficiently small for contraction not to occur.} 

We also point out that the stepsize restriction implicitly implies that $\gamma^{2}$ is larger than some constant factor multiplied by $M$. Further, for large $\gamma$ (for example  $\gamma \geq 5\sqrt{M}$) we have convergence for stepsizes independent of the size of $\gamma$ (for example $h < 1/8\sqrt{M}$). {The weakest stepsize restrictions for a kinetic Langevin discretization in the literature are those of \cite{sanz2021wasserstein}, which scale like $1/\gamma$. By contrast, for schemes which have stepsize restrictions independent of the friction parameter, $\gamma$ will have an increased stability property, which may be useful in practice.}
\end{remark}

 \begin{example} \textit{An example to illustrate the tightness of the restrictions on the stepsize $h$ and the restriction on the friction parameter $\gamma$.} We consider the anisotropic Gaussian distribution on $\mathbb{R}^{2}$ with potential $U: \mathbb{R}^{2} \mapsto \mathbb{R}$ given by $U(x,y) = \frac{1}{2}mx^{2} + \frac{1}{2}My^{2}$. By computing the eigenvalues of the transition matrix $P$ (for contraction) we can see for what values of $h$ contraction occurs. For BAO we have that
 \[
P = \begin{pmatrix} I_{2} - h^{2}Q & hI_{2}\\
-h\eta Q & \eta I_{2} 
\end{pmatrix}\text{, where }Q = \begin{pmatrix}
    m & 0 \\
    0 & M \end{pmatrix},
\]
with eigenvalues 
$\frac{1}{2}\left(1 + \eta - h^{2}\lambda \pm \sqrt{-4\eta + \left(-1 - \eta + h^{2}\lambda\right)^{2}}\right) \text{ for }\lambda = m,M$,
where $\eta = \exp{\{-\gamma h\}}$. For stability and contraction, it is necessary and sufficient that 
{
\[
\lambda_{\max}:=\max_{\lambda\in \{m,M\}}\left|\frac{1}{2}\left(1 + \eta - h^{2}\lambda \pm \sqrt{-4\eta + \left(-1 - \eta + h^{2}\lambda\right)^{2}}\right)\right| < 1.
\]
Due to the convexity of the absolute value function, it is clear that $|\frac{1}{2}(1+\eta-h^2 M)|<1$ is necessary for the stability condition to hold, so for any value of $\gamma$ and $m$, we need that $h\le \frac{2}{\sqrt{M}}$. 
Theorem \ref{Theorem:BAO} implies that the stepsize restriction $h\le \frac{1}{8 \sqrt{M}}$ suffices for stability for any $\gamma\ge 5\sqrt{M}$.
Further when $\gamma \geq 5\sqrt{M}$ we have that the asymptotic contraction rate for this Gaussian target simplifies to
$
c_{\mathcal{N}} = \frac{1}{2}\left(1- \eta + h^2m - \sqrt{\left(1 - \eta + h^2 m \right)^2 - 4h^2m}\right).
$
It can be shown that $4 c(h) > c_{\mathcal{N}}$ for  $\gamma\ge 5\sqrt{M}, h\le \frac{1}{8 \sqrt{M}}$.
It is shown in \cite{monmarche2020almost}[Proposition 4] that for the continuous dynamics the condition $\gamma>c\sqrt{M}$ for some constant $c>0$ is necessary to show contraction using a quadratic form argument similar to ours.\\
Despite this fact, for Gaussian targets, faster convergence rates can be achieved for BAO in the low-friction regime ($\gamma<5\sqrt{M}$). In particular, when we set $\gamma=2\sqrt{m}$, and use stepsize $h=\frac{1}{\sqrt{M}}$, with the notation $\rho=\frac{\sqrt{m}}{\sqrt{M}}$, the maximum norm becomes 
\begin{align*}
&\lambda_{\max}:=\\
&\frac{1}{2}\max\left(\left|e^{-2\rho}\pm \sqrt{e^{-4\rho}-4e^{-2\rho}}\right|, \left|1-\rho^2+e^{-2\rho} \pm \sqrt{(1+e^{-2\rho}-\rho^2)^2-4e^{-2\rho}}\right|\right). 
\end{align*}
It is not difficult to show with symbolic computing that for this choice of $\gamma$ and $h$, for any $0\le \rho\le 1$, we have 
\begin{equation}\label{eq:BAOacceleration}c_{\mathcal{N}}=1-\lambda_{\max}\ge \frac{3}{5}\rho=\frac{3}{5} \frac{\sqrt{m}}{\sqrt{M}}.\end{equation}
This is an accelerated convergence rate that is faster than what we could prove for general strongly convex and smooth potentials in Theorem \ref{Theorem:BAO}.
}
\end{example}

\begin{theorem}[OAB] \label{Theorem:OAB}
Assume $U$ is an $m$-strongly convex and $M$-$\nabla$Lipschitz potential. When $h < \min{\{\frac{1}{4\gamma},\frac{1-\eta}{\sqrt{6M}}}\}$, we have that for all initial conditions $\left(x_{0},v_{0}\right) \in \mathbb{R}^{2d}$ and $\left(\Tilde{x}_{0},\Tilde{v}_{0}\right)\in \mathbb{R}^{2d}$, and for any sequence of standard normal random variables $\left(\xi_{n}\right)_{n \in \mathbb{N}}$ the OAB chains $\left(x_{n},v_{n},\xi_{n}\right)_{n \in \mathbb{N}}$ and $\left(\Tilde{x}_{n},\Tilde{v}_{n},\xi_{n}\right)_{n \in \mathbb{N}}$ with initial conditions $\left(x_{0},v_{0}\right) \in \mathbb{R}^{2d}$ and $\left(\Tilde{x}_{0},\Tilde{v}_{0}\right)\in \mathbb{R}^{2d}$, respectively, satisfy
\[|\left|(x_{n} - \Tilde{x}_{n},v_{n} - \Tilde{v}_{n}\right)||_{a,b} \leq \left(1 - c\left( h\right)\right)^{\frac{n}{2}}||\left(x_{0} - \Tilde{x}_{0},v_{0} - \Tilde{v}_{0}\right)||_{a,b},
    \]
where $a = \frac{1}{M}$, $b = \frac{\eta h}{1-\eta}$ and $c\left(h\right) = \frac{\eta h^{2}m}{4\left(1 - \eta\right)}$.
\end{theorem}

Considering other splittings one could use the same techniques as above or we can use the contraction results of BAO and OAB to achieve a contraction result for the remaining permutations by writing 
$
    (\mathcal{ABO})^{n} = \mathcal{AB}(\mathcal{OAB})^{n-1}\mathcal{O}, (\mathcal{BOA})^{n} = \mathcal{B}(\mathcal{OAB})^{n-1}\mathcal{OA},
    (\mathcal{OBA})^{n} = \mathcal{O}(\mathcal{BAO})^{n-1}\mathcal{BA}$, and $(\mathcal{AOB})^{n} = \mathcal{AO}(\mathcal{BAO})^{n-1} \mathcal{B}.
$
However, by applying direct arguments as done for OAB and BAO one would achieve better preconstants. Let $\left(\Tilde{x}_{0}, \Tilde{v}_{0} \right) \in \mathbb{R}^{2d}$ and $\left(x_{0},v_{0}\right) \in \mathbb{R}^{2d}$ be two initial conditions for a synchronous coupling of sample paths of the ABO splitting and $\overline{x}_{0}:= \Tilde{x}_{0} - x_{0}$, $\overline{v}_{0}:= \Tilde{v}_{0} - v_{0}$. We use the notation $\Psi_{\rm{ABO}}$ to denote the one-step map of the ABO discretization with stepsize $h>0$, and equivalently for other operators {(omitting the stepsize in the argument, which is $h>0$ for all the one-step maps considered)}. We have that for $h < \min{\{\frac{1}{4\gamma},\frac{1-\eta}{\sqrt{6M}}}\}$
\begin{align*}
&||\Psi_{\rm{ABO}}\left(\Tilde{x}_{n},\Tilde{v}_{n}\right) -  \Psi_{\rm{ABO}}\left(x_{n},v_{n}\right)||^{2}_{a,b} = ||\left(\Psi_{\rm{ABO}}\right)^{n}\left(\Tilde{x}_{0},\Tilde{v}_{0}\right) -  \left(\Psi_{\rm{ABO}}\right)^{n}\left(x_{0},v_{0}\right)||^{2}_{a,b} \\
    &= ||\Psi_{\rm{O}} \circ \left(\Psi_{\rm{OAB}}\right)^{n-1} \circ \Psi_{\rm{AB}} \left(\Tilde{x}_{0},\Tilde{v}_{0}\right) - \Psi_{\rm{O}} \circ \left(\Psi_{\rm{OAB}}\right)^{n-1} \circ \Psi_{\rm{AB}}\left(x_{0},v_{0}\right)||^{2}_{a,b}\\
    &\leq 3\left(1 - c\left(h\right) \right)^{n-1} ||\Psi_{\rm{AB}}\left(\Tilde{x}_{0},\Tilde{v}_{0}\right) - \Psi_{\rm{AB}}\left(x_{0},v_{0}\right)||^{2}_{a,b}\\
    &\leq 9\left(1 - c\left(h\right)  \right)^{n-1}\left(\left(1 + 2h^{2}M^{2}a\right)||\overline{x}_{0}||^{2} + \left(h^{2} + a + 2h^{4} M^{2}a \right) ||\overline{v}_{0}||^{2} \right)\\
    &\leq 27\left(1 - c\left(h\right)  \right)^{n-1}||\left(\overline{x}_{0},\overline{v}_{0}\right)||^{2}_{a,b},
\end{align*}
where we have used the norm equivalence introduced in Section \ref{Sec:Quadratic_Norm}. The same method of argument can be used for the other first-order splittings.


\section{Higher-order splittings}

We now consider higher-order schemes which are obtained by the splittings introduced in Section \ref{sec:splittings}. These schemes are weak order two and they are symmetric in the order of the operators, with repeated operators corresponding to multiple steps with half the stepsize. We will focus our attention on two popular splittings which are BAOAB and ABOBA (or OBABO) as in \cite{leimkuhler2013rational}. Due to the fact that the modified Euclidean norms developed in the previous section are different for different first-order splitting, we aren't able to simply compose the results of say OBA and ABO to obtain contraction of OBABO. First we consider the BAOAB discretization, where we denote a BAOAB chain with initial condition $\left(x_{0},v_{0}\right) \in \mathbb{R}^{2d}$ by $\left(x_{n},v_{n},\xi_{n}\right)_{n \in \mathbb{N}}$, which are defined by the update $\mathcal{BAOAB}$ (\ref{eq:BAO}) and $\left(\xi_{n}\right)_{n \in \mathbb{N}}$ are independent Gaussian random variables.

\begin{theorem}[BAOAB] \label{Theorem:BAOAB} Assume $U$ is an $m$-strongly convex and $M$-$\nabla$Lipschitz potential. When  $h \leq \frac{1-\eta}{2\sqrt{M}}$, we have that for all initial conditions $\left(x_{0},v_{0}\right) \in \mathbb{R}^{2d}$ and $\left(\Tilde{x}_{0},\Tilde{v}_{0}\right)\in \mathbb{R}^{2d}$, and for any sequence of standard normal random variables $\left(\xi_{n}\right)_{n \in \mathbb{N}}$ the BAOAB chains $\left(x_{n},v_{n},\xi_{n}\right)_{n \in \mathbb{N}}$ and $\left(\Tilde{x}_{n},\Tilde{v}_{n},\xi_{n}\right)_{n \in \mathbb{N}}$ with initial conditions $\left(x_{0},v_{0}\right) \in \mathbb{R}^{2d}$ and $\left(\Tilde{x}_{0},\Tilde{v}_{0}\right)\in \mathbb{R}^{2d}$, respectively, satisfy
\[||\left(x_{n} - \Tilde{x}_{n},v_{n} - \Tilde{v}_{n}\right)||_{a,b} \leq 7\left(1 - c\left(h\right)\right)^{\frac{n-1}{2}}||\left(x_{0} - \Tilde{x}_{0},v_{0} - \Tilde{v}_{0}\right)||_{a,b},
    \]
where $a = \frac{1}{M}$ and $b = \frac{h}{1-\eta}$ and
$
c\left(h\right) = \frac{h^{2}m}{4\left(1- \eta\right)}.
$
\end{theorem}

Next, we consider the OBABO discretization which has been studied in the recent work \cite{monmarche2021high}. We denote a OBABO chain with initial condition $\left(x_{0},v_{0}\right) \in \mathbb{R}^{2d}$ by $\left(x_{n},v_{n},\xi_{n}\right)_{n \in \mathbb{N}}$, which are defined by the update $\mathcal{OBABO}$ (\ref{eq:BAO}) and $\left(\xi_{n}\right)_{n \in \mathbb{N}}$ are independent Gaussian random variables.

\begin{theorem}[OBABO]\label{Theorem:OBABO}
 Assume $U$ is an $m$-strongly convex and $M$-$\nabla$Lipschitz potential. When  $h < \frac{1 - \eta}{\sqrt{4M}}$, we have that for all initial conditions $\left(x_{0},v_{0}\right) \in \mathbb{R}^{2d}$ and $\left(\Tilde{x}_{0},\Tilde{v}_{0}\right)\in \mathbb{R}^{2d}$, and for any sequence of standard normal random variables $\left(\xi_{n}\right)_{n \in \mathbb{N}}$ the OBABO chains $\left(x_{n},v_{n},\xi_{n}\right)_{n \in \mathbb{N}}$ and $\left(\Tilde{x}_{n},\Tilde{v}_{n},\xi_{n}\right)_{n \in \mathbb{N}}$ with initial conditions $\left(x_{0},v_{0}\right) \in \mathbb{R}^{2d}$ and $\left(\Tilde{x}_{0},\Tilde{v}_{0}\right)\in \mathbb{R}^{2d}$, respectively, satisfy
\[||\left(x_{n} - \Tilde{x}_{n},v_{n} - \Tilde{v}_{n}\right)||_{a,b} \leq 7\left(1 - c\left(h\right)\right)^{\frac{n-1}{2}}||\left(x_{0} - \Tilde{x}_{0},v_{0} - \Tilde{v}_{0}\right)||_{a,b},
    \]
where $a = \frac{1}{M}$, $b = \frac{h}{1-\eta}$ and
$
c\left(h\right) = \frac{h^{2}m}{4\left(1 - \eta\right)}.
$
\end{theorem}

\begin{remark}
In \cite{dalalyan2020sampling} it is shown that the continuous dynamics converges with a rate of $\mathcal{O}(m/\gamma)$. There is a major difference in terms of contraction rate for large $\gamma$ between the rates achieved by BAOAB and OBABO and the continuous dynamics. As in the limit $\gamma \to \infty$ in BAOAB and OBABO one has convergence rates of $\mathcal{O}(h^{2}m)$, whereas the contraction rate of the continuous dynamics converges to zero. A time-rescaling of $(t' = \gamma t)$ is required to derive the overdamped Langevin dynamics.
\end{remark}

\begin{remark}
    In Theorem \ref{Theorem:BAOAB} and Theorem \ref{Theorem:OBABO} we have a prefactor of $7$ because we have converted the problem of contraction into proving a simpler problem with one gradient evaluation. {To be specific, for BAOAB using the relation $(\mathcal{BAOAB})^{n} = \mathcal{BAO}\left(\mathcal{ABAO}\right)^{n-1}\mathcal{AB}$, we prove contraction for $\mathcal{ABAO}$ and similarly for OBABO.} The prefactor comes from the remaining terms $\mathcal{BAO}$ and $\mathcal{AB}$.
\end{remark}

\begin{remark}
    In \cite{monmarche2022hmc} they analyze Hamiltonian Monte Carlo as $\mathcal{O}\left(\mathcal{ABA}\right)^{L}\mathcal{O}$ for $L$ leapfrog steps. In \cite{monmarche2022hmc} a similar norm is used to study Hamiltonian Monte Carlo, however, they obtain stepsize restrictions of at least $\mathcal{O}\left(m/L^{3/2}\right)$.
\end{remark}
\section{The stochastic exponential Euler scheme}
For the stochastic exponential Euler scheme given in the update rule \ref{eq:SES} we can couple two trajectories which have common noise $\left(\zeta_{n},\omega_{n}\right)_{n \in \mathbb{N}}$, then we can obtain contraction rates by the previously introduced methods. We denote an SES chain with initial condition $\left(x_{0},v_{0}\right) \in \mathbb{R}^{2d}$ by {$\left(x_{n},v_{n},\zeta_{n},\omega_{n}\right)_{n \in \mathbb{N}}$}, which are defined by the update SES (\ref{eq:SES}).

\begin{theorem}[Stochastic Euler Scheme]\label{Theorem:SES}
Assume $U$ is an $m$-strongly convex and $M$-$\nabla$Lipschitz potential. When $\gamma \geq 5\sqrt{M}$ and $h \leq \frac{1}{2\gamma}$, we have that for all initial conditions $\left(x_{0},v_{0}\right) \in \mathbb{R}^{2d}$ and $\left(\Tilde{x}_{0},\Tilde{v}_{0}\right)\in \mathbb{R}^{2d}$, and for any sequence of  random variables $\left({\zeta_{n},\omega_{n}}\right)_{n \in \mathbb{N}}$ {defined by \eqref{eq:SES_noise}}, the SES chains $\left(x_{n},v_{n},{\zeta_{n},\omega_{n}}\right)_{n \in \mathbb{N}}$ and $\left(\Tilde{x}_{n},\Tilde{v}_{n},{\zeta_{n},\omega_{n}}\right)_{n \in \mathbb{N}}$ with initial conditions $\left(x_{0},v_{0}\right) \in \mathbb{R}^{2d}$ and $\left(\Tilde{x}_{0},\Tilde{v}_{0}\right)\in \mathbb{R}^{2d}$, respectively, satisfy
\[||\left(x_{n} - \Tilde{x}_{n},v_{n} - \Tilde{v}_{n}\right)||_{a,b} \leq \left(1 - c\left(h\right)\right)^{\frac{n}{2}}||\left(x_{0} - \Tilde{x}_{0},v_{0} - \Tilde{v}_{0}\right)||_{a,b},
    \]
where $a = \frac{1}{M}$, $b = \frac{1}{\gamma}$ and
$
c\left(h\right) = \frac{mh}{4\gamma}.
$
\end{theorem}

\section{Overdamped Limit}
We will now compare and analyze how the different schemes behave in the high-friction limit, where we first start with the first-order schemes. It is a desirable property that the high-friction limit is a discretization of the overdamped dynamics, therefore if a user of such a scheme sets the friction parameter $\gamma$ large, they will not suffer from the $\mathcal{O}(1/\gamma)$ scaling of the convergence rate. We will call schemes with this desirable property $\gamma$-limit convergent (GLC), out of the schemes we have analysed it is only BAOAB and OBABO which are GLC.
\subsection{BAO}
If we consider the update rule of the BAO scheme
\begin{align*}
    x_{n+1} = x_{n} + h\left(v_{n} - h \nabla U(x_{n})\right), \quad v_{n+1} = \eta v_{n} - h\eta \nabla U(x_{n}) + \sqrt{1 - \eta^{2}}\xi_{n+1},
\end{align*}
and take the limit as $\gamma \to \infty$ we obtain
\begin{align*}
    x_{n+1} &= x_{n} - h^{2} \nabla U(x_{n}) + h \xi_{n},
\end{align*}
which is simply the Euler-Maruyama scheme with stepsize $h^{2}/2$ for potential $\Tilde{U} := 2U$, which imposes stepsize restrictions which are consistent with our analysis. Further, if we take the limit of the contraction rate and the modified Euclidean norm we have
\[
\lim_{\gamma \to \infty} c\left(h\right) = \frac{h^{2}m}{4}, \qquad \lim_{\gamma \to \infty} ||x||^{2} + 2b\langle x,v \rangle + a||v||^{2} = ||x||^{2} + 2h \langle x,v \rangle + \frac{1}{M}||v||^{2},
\]
which is again consistent with the convergence rates achieved in Section \ref{sec:conv_overdamped} and the norm is essentially the Euclidean norm when considered on the overdamped process as $\overline{v} = 0$. Due to the fact that the potential is rescaled in the limit, this is not a discretization of the {correct} overdamped dynamics.

\subsection{OAB}
If we consider the update rule of the OAB scheme
\begin{align*}
    x_{n+1} &= x_{n} + h\eta v_{n} + h \sqrt{1 - \eta^{2}}\xi_{n+1},\\
    v_{n+1} &= \eta v_{n} \sqrt{1 - \eta^{2}}\xi_{n+1} - h\eta \nabla U(x_{n} + h\eta v_{n} + h \sqrt{1 - \eta^{2}}\xi_{n+1}),
\end{align*}
and take the limit as $\gamma \to \infty$ we obtain the update rule $x_{n+1} = x_{n} + h \xi_{n+1}$, therefore the overdamped limit is not inherited by the scheme and further we do not expect contraction. This is consistent with our analysis of OAB and our contraction rate which vanishes in the high-friction limit.
\subsection{BAOAB}
If we consider the update rule of the BAOAB scheme
\begin{align*}
    x_{n+1} &= x_{n} + \frac{h}{2}\left(1 + \eta\right)v_{n} - \frac{h^{2}}{4}\left(1 + \eta\right) \nabla U(x_{n}) + \frac{h}{2}\sqrt{1 - \eta^{2}}\xi_{n+1},\\
    v_{n+1} &= \eta \left(v_{n} - \frac{h}{2}\nabla U(x_{n})\right) + \sqrt{1 - \eta^{2}}\xi_{n+1} - \frac{h}{2}\nabla U(x_{n+1}),
\end{align*}
and take the limit as $\gamma \to \infty$ we obtain
\begin{align*}
    x_{n+1} &= x_{n} - \frac{h^{2}}{2} \nabla U(x_{n}) + \frac{h}{2}\left(\xi_{n} + \xi_{n+1}\right),
\end{align*}
which is simply the LM scheme with stepsize $h^{2}/2$ (as originally noted in \cite{leimkuhler2013rational}), which imposes stepsize restrictions $h^{2} \leq 2/M$ and hence consistent with our analysis. Further, if we take the limit of the contraction rate and the modified Euclidean norm we have
\[
\lim_{\gamma \to \infty} c\left(h\right) = \frac{h^{2}m}{4}, \qquad \lim_{\gamma \to \infty} ||x||^{2} + 2b\langle x,v \rangle + a||v||^{2} = ||x||^{2} + 2h \langle x,v \rangle + \frac{1}{M}||v||^{2},
\]
which is again consistent with the convergence rates achieved in Section \ref{sec:conv_overdamped} and the modified Euclidean norm is essentially the Euclidean norm when considered on the overdamped process as $\overline{v} = 0$.

\subsection{OBABO}
If we consider the update rule of the OBABO scheme 
\begin{align*}
    x_{n+1} &= x_{n} + h\eta v_{n} + h\sqrt{1 - \eta^2}\xi_{1,n+1} - \frac{h^{2}}{2}\nabla U(x_{n}),\\
    v_{n+1} &= \eta\left(\eta v + \sqrt{1 - \eta^2}\xi_{1,n+1} - \frac{h}{2}\nabla U(x_{n}) - \frac{h}{2}\nabla U(x_{n+1})\right) + \sqrt{1 - \eta^{2}}\xi_{2,n+1},
\end{align*}
where ($\eta = \exp{\left(-\gamma h /2\right)}$) and for ease of notation in the above scheme and we have labelled the two noises of one step $\xi_{1}$ and $\xi_{2}$. Now we take the limit as $\gamma \to \infty$ we obtain
\begin{align*}
    x_{n+1} &= x_{n} -\frac{h^{2}}{2}\nabla U(x_{n}) + h \xi_{n+1},
\end{align*}
which is the Euler-Maruyama scheme for overdamped Langevin with stepsize $h^{2}/2$, which has convergence rate $\mathcal{O}\left(h^{2}m\right)$. Hence consistent with our analysis of OBABO and our contraction rate which tends towards $h^{2}m/4$ in the high-friction limit.

\subsection{SES}
If we consider the limit as $\gamma \to \infty$ of the scheme (\ref{eq:SES}) we obtain the update rule $x_{n+1} = x_{n}$ and therefore the overdamped limit is not inherited by the scheme and further we do not expect contraction. Hence consistent with our analysis of the stochastic Euler scheme as the contraction rate tends to zero in the high-friction limit.

\section{Asymptotic Bias of BAOAB}\label{sec:bias}

There are results for the asymptotic bias of OBABO available in \cite{monmarche2021high,monmarche2022hmc} and for the SES in \cite{sanz2021wasserstein} which can easily be combined with our results. However there are no results for BAOAB available in the literature, we will provide asymptotic bias estimates for this scheme. We will do this with the aim of achieving bias estimates which remain finite in the high-friction limit to show that BAOAB and GLC schemes remain useful for sampling even though they deviate from the continuous dynamics.

We define the solution map $\mathcal{H}$ to have update rule
\begin{equation}\label{eq:H_step}
     \mathcal{H}: (x,v) \to \phi_{h}(x,v),
\end{equation}
where $\phi_{h}(x,v)$ is the solution to the ODE
\begin{align*}
    dX_{t} &= V_{t}dt, \qquad dV_{t} = -\nabla U(X_{t})dt,
\end{align*}
initialized at $(X_{0},V_{0}) := (x,v) \in \mathbb{R}^{2d}$ at time $h > 0$. Since the BAOAB scheme converges faster than the continuous dynamics, if we compare BAOAB to the underdamped Langevin dynamics in the high-friction regime, we will get discretization bounds which diverge as $\gamma \to \infty$ and the stepsize is kept constant. We instead compare BAOAB to a scheme which performs exact Hamiltonian dynamics for half a step, followed by an OU-process for a full step, followed by Hamiltonian dynamics for half a step. We call this scheme the HOH scheme with the update rule given by $\mathcal{HOH}$ with $\mathcal{H}$ defined in \eqref{eq:H_step}. This process exactly preserves the invariant measure and is a more accurate approximation of the BAOAB scheme than \eqref{eq:underdamped_langevin}. 

\begin{remark}
In \cite{monmarche2021high} and \cite{monmarche2022hmc}, the authors provide bias estimates for the OBABO and OABAO schemes. In their analysis of these schemes, they use the fact that the ``O" step preserves the invariant measure and doesn’t increase Wasserstein distance.  They are then able to exploit $L^{2}$-accuracy results of the embedded Verlet integrators BAB and ABA in their analysis.
\end{remark}

\begin{proposition}\label{prop:BAOAB_local_error}
    Consider an HOH scheme initialized at $(x,v) \in \mathbb{R}^{2d}$ and a BAOAB scheme initialized at $(x',v') \in \mathbb{R}^{2d}$ with synchronously coupled Gaussian increments and stepsize $h > 0$, then we define $(\Delta_{x},\Delta_{v}) := \Psi_{\textnormal{HOH}}(x,v,h) - \Psi_{\textnormal{BAOAB}}(x',v',h)$ with shared noise $\xi \sim \mathcal{N}(0,I_{d})$. We assume that $h <\frac{1-\eta}{2\sqrt{M}}$ and Assumptions \ref{assum:G_Lipschitz} - \ref{assum:convex} on the potential, then we have that
    \begin{equation*}
        \begin{split}
            &\|\Delta_{x}\|_{L^{2}} \leq  \left(1 + (1+\eta)\frac{h^{2}}{4}M\right)\|x - x'\|_{L^{2}} + \frac{h}{2}(1+\eta)\|v-v'\|_{L^{2}} + \frac{3h^{3}M\sqrt{d}}{8},\\
            &\Delta_{v}= \left(\eta I_{d} - \frac{h^{2}(1+\eta)}{4}Q_{2}\right)(v-v') +\left(-\frac{h\eta}{2}Q_{1} - \frac{h}{2}Q_{2} + \frac{h^{3}(1+\eta)}{8}Q_{2}Q_{1}\right)(x-x') \\
    &+ \epsilon_{v},
            \intertext{where $\|\epsilon_{v}\|_{L^{2}} \leq 2h^{2}M\sqrt{d}$. Further if we assume Assumption \ref{assum:H_Lipschitz} we have}
            \Delta_{v}&= \left(\eta I_{d} - \frac{h^{2}(1+\eta)}{4}Q_{2}\right)(v-v') +\left(-\frac{h\eta}{2}Q_{1} - \frac{h}{2}Q_{2} + \frac{h^{3}(1+\eta)}{8}Q_{2}Q_{1}\right)(x-x') \\
    &+ \epsilon_{v} + h^{2}\sqrt{1-\eta^{2}}A(x,x')\xi,
        \end{split}
    \end{equation*}
where $m I_{d} \prec Q_{1}, Q_{2} \prec MI_{d}$, $\epsilon_{v} \in \mathbb{R}^{2d}$, $A(x,x') \in \mathbb{R}^{d \times d}$ and $\|\cdot \|_{L^{2}} := (\mathbb{E}\| \cdot \|^{2})^{1/2}$, with $\|\epsilon_{v}\|_{L^{2}} \leq 2h^{3}\sqrt{d}(M^{3/2}+M_{1}\sqrt{d})$. Further we have that $\|A(x,x')\xi\|_{L^{2}} \leq \frac{3}{8}M\sqrt{d}$.
\end{proposition}
\begin{proof}
    The proof is found in Appendix \ref{Appendix:Bias}.
\end{proof}
\begin{proposition}\label{prop:second_kernel_estimate}
Consider an HOH scheme, $(x_{i},v_{i})_{i \in \mathbb{N}}$ and a BAOAB scheme $(x'_{i},v'_{i})_{i\in \mathbb{N}}$ initialized at $(x_{0},v_{0}) = (x'_{0},v'_{0}) = (x,v)\sim \pi$ in $\mathbb{R}^{2d}$ with synchronously coupled Gaussian increments and stepsize $h<\frac{1-\eta}{2\sqrt{M}}$, for $l \in \mathbb{N}$ we define $(\Delta^{l}_{x},\Delta^{l}_{v}) :=  (x_{l}-x'_{l},v_{l}-v'_{l})$. For $a = 1/M$ and $b = h/(1-\eta)$ we have that under Assumptions \ref{assum:G_Lipschitz}-\ref{assum:convex}, for any $l\ge 1$,
\[
\|(\Delta^{l}_{x},\Delta^{l}_{v})\|_{L^{2},a,b} \leq 4400 \frac{M}{m} \sqrt{d}h.
\]
Additionally, if Assumption \ref{assum:H_Lipschitz} is satisfied, we have for any $l\ge 1$,
\[
\|(\Delta^{l}_{x},\Delta^{l}_{v})\|_{L^{2},a,b}  \leq 
1500 \frac{\sqrt{M}}{m}\left(4\sqrt{Md}+3\frac{M_1}{M} d\right)h(1-\eta).
\]
\end{proposition}
\begin{proof}
    The proof is found in Appendix \ref{Appendix:Bias}.
\end{proof}
\begin{remark}
    In Proposition \ref{prop:second_kernel_estimate} we provide $L^{2}$ error estimates over $l \in \mathbb{N}$ steps, this allows one to go from order $h^{5/2}$ local error to order $h^{2}$ global error, as the $h^{5/2}$ term in the local error estimate is due to independent Gaussian increments. Similarly, the strong convergence of numerical solutions of SDEs only loses an order of $1/2$ accuracy (see \cite{MR2069903}[Theorem 1.1.1]), for example, the Euler-Maruyama scheme or the UBU scheme in \cite{sanz2021wasserstein}.
\end{remark}

\begin{theorem}\label{theorem:asymptotic_bias}
    Consider a BAOAB scheme with stepsize $h < \frac{1-\eta}{2\sqrt{M}}$ and invariant measure $\pi_{h}$ on $\mathbb{R}^{2d}$, with target measure $\pi$ on $\mathbb{R}^{2d}$. For $a = 1/M$ and $b = h/(1-\eta)$ we have that under Assumptions \ref{assum:G_Lipschitz}-\ref{assum:convex} 
    \[
    \mathcal{W}_{2,a,b}(\pi,\pi_{h}) \leq 4400 \frac{M}{m} \sqrt{d} h.
    \]
 Additionally, if Assumption \ref{assum:H_Lipschitz} is satisfied we have
    \begin{align*}
            \mathcal{W}_{2,a,b}(\pi,\pi_{h}) &\leq 1500 \frac{\sqrt{M}}{m}\left(4\sqrt{Md}+3\frac{M_1}{M} d\right)h(1-\eta).
        \end{align*}
\end{theorem}
\begin{proof}
    Let $P_{h}$ denote the transition kernel of the BAOAB scheme, which has invariant measure $\pi_{h}$. Then we have that for $l \in \mathbb{N}$
    \begin{align*}
        \mathcal{W}_{2,a,b}(\pi,\pi_{h}) \leq \mathcal{W}_{2,a,b}(\pi P^{l}_{h},\pi_{h}) + \mathcal{W}_{2,a,b}(\pi P^{l}_{h},\pi).
    \end{align*}
    We now estimate each of the terms on the right-hand side separately. We have by Theorem \ref{Theorem:BAOAB} 
    \begin{align*}
        \mathcal{W}_{2,a,b}(\pi P^{l}_{h},\pi_{h}) = \mathcal{W}_{2,a,b}(\pi P^{l}_{h},\pi_{h}P^{l}_{h}) \leq 7(1-c(h))^{(l-1)/2}\mathcal{W}_{2,a,b}(\pi,\pi_{h}).
    \end{align*}
    This term tends to zero as $l$ tends to infinity.
    By Proposition \ref{prop:second_kernel_estimate}, and the definition of the Wasserstein distance, 
    we can bound the term $\mathcal{W}_{2,a,b}(\pi P^{l}_{h},\pi)$ uniformly for any $l$, hence our claims follow.
\end{proof}
Note that the constraints on the stepsize and friction parameter are the same as for our contraction result i.e. close to the stability threshold. A consequence of the asymptotic bias results together with the contraction results is that when the potential satisfies Assumptions \ref{assum:G_Lipschitz}-\ref{assum:convex} the BAOAB scheme requires $\mathcal{O}(d^{1/2}/\epsilon)$ steps to reach an accuracy $\epsilon > 0$ in Wasserstein distance from the target. If the potential further satisfies Assumption \ref{assum:H_Lipschitz} this can be improved to $\mathcal{O}(d^{1/2}/\epsilon^{1/2})$ steps to reach an accuracy $\epsilon > 0$ in Wasserstein distance from the target. Perhaps under stronger assumptions like the strongly Hessian Lipschitz assumption of \cite{chen2023does}, the dimension dependency can be improved. We remark that the strongly Hessian Lipschitz assumption of \cite{chen2023does} implies Assumption \ref{assum:H_Lipschitz}.

\section{Discussion}
\begin{figure}
	\centering
	\begin{subfigure}[b]{0.32\textwidth}
		\centering
		\includegraphics[width=\textwidth]{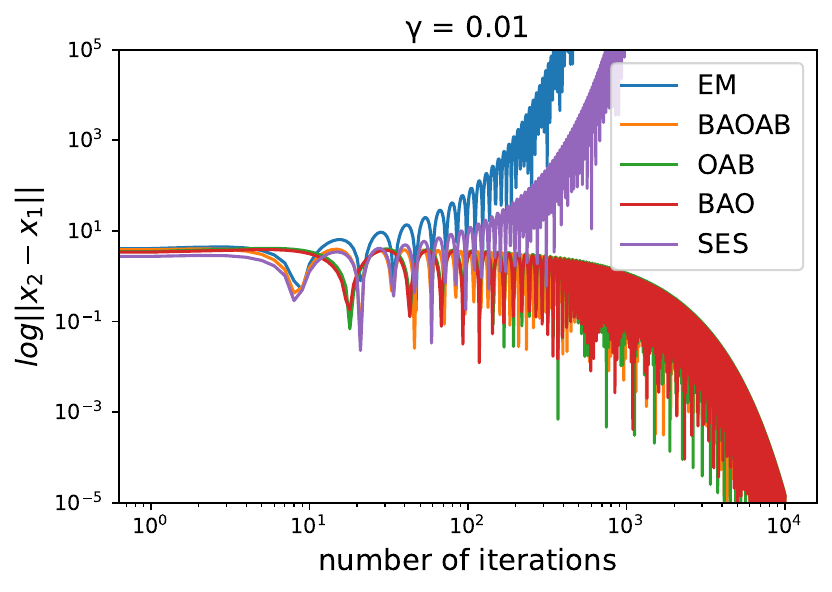}
		\caption{Low friction}
		
	\end{subfigure}
	\hfill
	\begin{subfigure}[b]{0.32\textwidth}
		\centering
		\includegraphics[width=\textwidth]{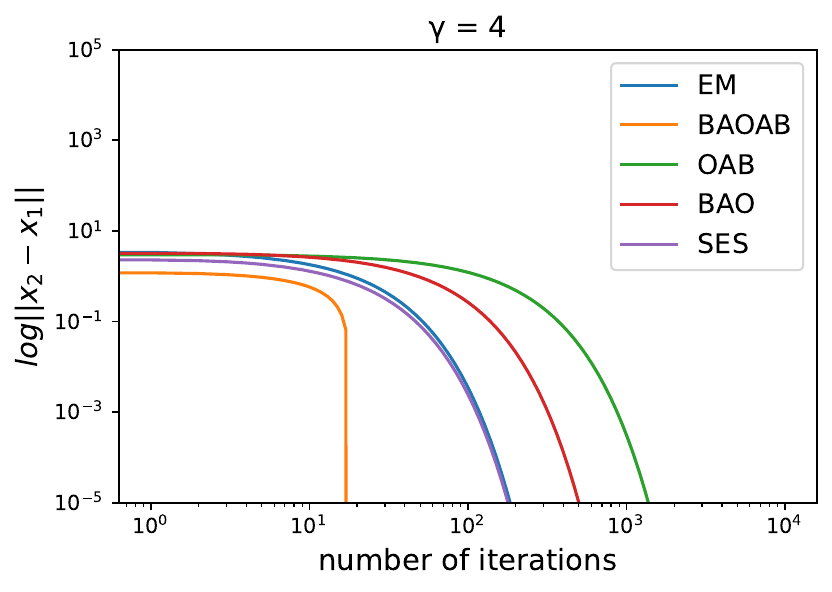}
		\caption{}
		
	\end{subfigure}
	\hfill
	\begin{subfigure}[b]{0.32\textwidth}
		\centering
		\includegraphics[width=\textwidth]{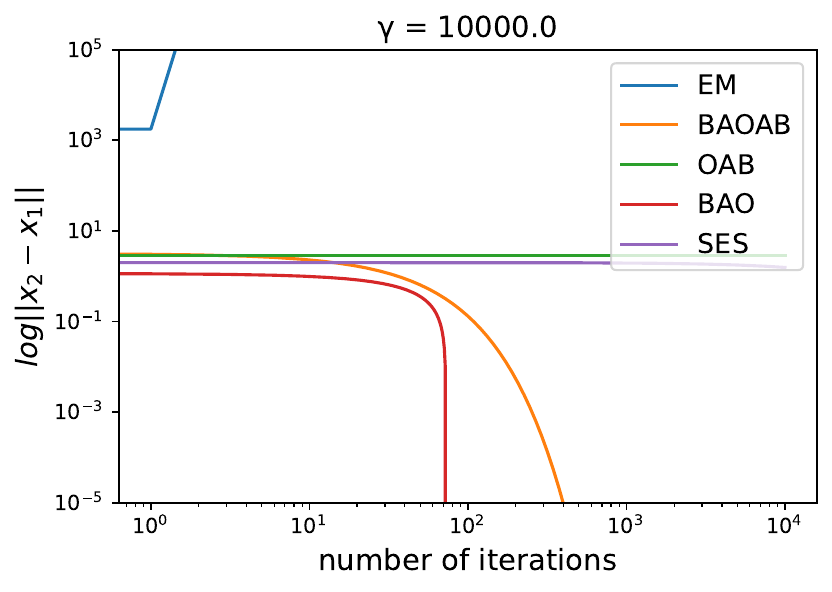}
		\caption{High friction}
		
	\end{subfigure}
	\caption{Contraction of two kinetic Langevin trajectories $x_{1}$ and $x_{2}$ with initial conditions $[-1,-1]$ and $[1,1]$ for a $2$-dimensional standard Gaussian with stepsize $h = 0.25 = 1/4\sqrt{M}$.}
    \label{Fig:1}
\end{figure}
We tested our observations numerically in Figure  \ref{Fig:1} with a $2$-dimensional standard Gaussian. Figure  \ref{Fig:1} is consistent with our analysis that all schemes are stable when $\gamma \approx 4\sqrt{M}$ and in the high-friction regime EM, OAB and SES behave poorly compared to BAOAB and BAO. In the low-friction regime again EM and SES perform poorly compared to the other schemes. 

In \cite{dalalyan2020sampling} it is shown that the optimal convergence rate for the continuous time dynamics is $\mathcal{O}(m/\gamma)$, therefore our contraction rates are consistent up to a constant for the discretizations, however for some of the schemes considered for example BAOAB and OBABO we have that the scheme inherits convergence to the overdamped Langevin dynamics (without time rescaling) and this is reflected in our convergence rate estimates. For MCMC applications, our estimates of convergence rate are independent of $\gamma$.  In particular, in the case of $\gamma$-limit-convergence methods, our convergence guarantees are valid for large friction. The robustness with respect to the friction parameter is shown in Figure  \ref{Fig:1}. BAOAB and OBABO do not suffer from slow convergence in the limit and we do not expect a large bias because they converge to a consistent overdamped Langevin dynamics numerical scheme. {We provide asymptotic bias estimates for the BAOAB scheme which remain finite in the high-friction limit and provide non-asymptotic guarantees for this scheme.} We recommend these schemes in the context of sampling due to their robustness with respect to the choice of friction parameter and increased stability.

The constants in our arguments can be improved by sharper bounds and a more careful analysis, but the restriction on $\gamma$ is consistent with other works on synchronous coupling for the continuous time Langevin diffusions \cite{bolley2010trend,cheng2018underdamped,dalalyan2020sampling,deligiannidis2021randomized,zajic2019non}. Further it is shown in \cite{monmarche2020almost}[Proposition 4] that the continuous time process yields Wasserstein contraction by synchronous coupling for all $M$-$\nabla$Lipschitz and $m$-strongly convex potentials $U$ if and only if $M - m < \gamma (\sqrt{M} + \sqrt{m})$ for the norms that we considered. This condition when $M$ is much larger than $m$ is $\mathcal{O}(\sqrt{M})$. It may be possible to achieve convergence rates for small $\gamma$, by using a more sophisticated argument like that of \cite{eberle2019couplings}. Using a different Lyapunov function or techniques may lead to being able to extend these results to all $\gamma > 0$ \cite{durmus2021uniform,qin2022geometric}, but this is beyond the scope of this paper.

The restrictions on the stepsize $h$ are tight for the optimal contraction rate for EM and BAO and hence result in stability conditions of $\mathcal{O}\left(1/\gamma\right)$ for EM and SES. Also, we have shown BAO, OBA, AOB, BAOAB and OBABO have convergence guarantees for stepsizes $\mathcal{O}(1/\sqrt{M})$ and BAOAB and OBABO have the desirable GLC property which is not common amongst the schemes we studied. For the choice of parameters which achieve optimal contraction rate, we derive $\mathcal{O}(m/M)$ rates of contraction, which are sharp up to a constant and we achieve this for every scheme that we studied.

{The constants for the discretization analysis for BAOAB can be improved. However, the dimension-dependence of the non-asymptotic guarantees is in accordance with other numerical integrators for kinetic Langevin dynamics without the use of randomized midpoint methods (see \cite{shen2019randomized}).}

\section*{Acknowledgments}
The authors would like to thank Kostas Zygalakis for helpful comments on this work. The authors would also like to thank the anonymous referees for their feedback and suggestions, which have
improved the quality of the paper. The authors acknowledge the support of the Engineering and Physical Sciences Research Council Grant EP/S023291/1 (MAC-MIGS
Centre for Doctoral Training).

\bibliographystyle{siamplain}
\bibliography{references}
\appendix
\section{Proofs}

\begin{proof}[Proof of Theorem \ref{Theorem:EM}]
We will denote two synchronous realisations of EM as $\left(x_{j},v_{j}\right)$ and $\left(\Tilde{x}_{j}, \Tilde{v}_{j}\right)$ for $j \in \mathbb{N}$. Now we will denote $\overline{x}_{j} := \left(\Tilde{x}_{j} - x_{j}\right)$, $\overline{v}_{j} = \left(\Tilde{v}_{j} - v_{j}\right)$ and $\overline{z}_{j} = \left(\overline{x}_{j}, \overline{v}_{j}\right)$, where $\overline{z}_{j} = \left(\overline{x}_{j}, \overline{v}_{j}\right)$ for $j = n,n +1$ for $n \in \mathbb{N}$. We have the following update rule for $\overline{z}_{n}$
\begin{align*}
    \overline{x}_{n+1} = \overline{x}_{n} + h \overline{v}_{n}, \qquad  \overline{v}_{n+1} = \overline{v}_{n} - \gamma h \overline{v}_{n} -hQ\overline{x}_{n},
\end{align*}
where by mean value theorem we define $Q = \int^{1}_{t = 0}\nabla^{2}U(\Tilde{x}_{n} + t(x_{n} - \Tilde{x}_{n}))dt$, then $\nabla U(\Tilde{x}_{n}) - \nabla U(x_{n}) = Q\overline{x}_{n}$. 
One can show that in the notation of equation \eqref{eq:contraction_matrix_form} we have 
\begin{equation}\label{eq:P_matrix2}
    P = \begin{pmatrix} I_{d} & hI_{d}\\
 -hQ & \left(1- \gamma h\right)I_{d} 
\end{pmatrix},
\end{equation}
and therefore for Euler-Maruyama (using the notation of equation (\ref{eq:contraction_matrix})) 
\begin{align*}
    A &= -c\left(h\right)I_{d}  + 2bhQ - h^{2}aQ^{2},\\
    B &=  -bc\left(h\right)  I_{d}  + h\left((b\gamma - 1)I_{d} + (a + h(b-a\gamma))Q\right), \\
    C &= \left(-c\left(h\right) a  + h\left(2a\gamma - 2b - h(1 - 2b\gamma + a \gamma^{2}) \right)\right)I_{d}.
\end{align*}
We now invoke Proposition \ref{Prop:PD} as $A$, $B$ and $C$ commute as they are all polynomials in $Q$. $A$ is positive definite if and only if all its eigenvalues are positive. We note that the eigenvalues of $A$ are precisely $P_{A}(\lambda) := -c\left(h\right) + h\left( 2b\lambda - ha\lambda^{2}\right)$, where $\lambda$ are the eigenvalues of $Q$, where $m \leq \lambda \leq M$. We wish to show that $P_{A}(\lambda) > 0$ for all $\lambda \in [m,M]$. This is equivalent to 
 \[
     \frac{P_{A}\left(\lambda\right)}{h} = -\frac{m}{2\gamma} + \frac{2\lambda}{\gamma} - \frac{h\lambda^{2}}{M} \geq \lambda\left(-\frac{1}{2\gamma} + \frac{2}{\gamma} - h\right) > 0,
 \]
 which is satisfied when $h < 1/\gamma$. Hence we have that $A \succ 0$. Now it remains to prove that $AC - B^{2} \succ 0$, where $AC - B^{2}$ is a polynomial of $Q$, which we denote $P_{AC-B^{2}}(Q)$. Hence {it} has eigenvalues dictated by the eigenvalues $\lambda$ of $Q$. That is the eigenvalues of $AC - B^{2}$ are $P_{AC-B^{2}}(\lambda)$ for $\lambda$ an eigenvalue of $Q$. Considering $P_{AC - B^{2}}$ we have
 \begin{align*}
      &\frac{P_{AC-B^{2}}\left(\lambda\right)}{h^{2}\lambda} = \frac{4}{M} -\frac{4}{\gamma^2} + \frac{m^2}{4 \gamma^2 M
   \lambda}-\frac{m^2}{4 \gamma^4 \lambda}+\frac{m}{\gamma^2
   \lambda}-\frac{m}{M
   \lambda}-\frac{\lambda}{M^2} + h^2 \left(\frac{\lambda}{M}-\frac{\lambda}{\gamma^2}\right)\\
   &+h
   \left(\frac{2}{\gamma} -\frac{m}{\gamma M} -\frac{m}{2 \gamma \lambda} + \frac{m}{\gamma^3}+\frac{m \lambda}{2 \gamma M^2}+\frac{\gamma m}{2
   M \lambda}-\frac{2
   \gamma}{M}\right) > \frac{1}{M} - h\frac{2\gamma}{M} > 0,
 \end{align*}
where we have used the fact that $\gamma \geq 2\sqrt{M}$ and $h < \frac{1}{2\gamma}$ and hence $AC - B^{2} \succ 0$.
\end{proof}


\begin{proof}[Proof of Theorem \ref{Theorem:BAO}]
We will denote two synchronous realisations of BAO as $\left(x_{j},v_{j}\right)$ and $\left(\Tilde{x}_{j}, \Tilde{v}_{j}\right)$ for $j \in \mathbb{N}$. Now we will denote $\overline{x}_{j} := \left(\Tilde{x}_{j} - x_{j}\right)$, $\overline{v}_{j} = \left(\Tilde{v}_{j} - v_{j}\right)$ and $\overline{z}_{j} = \left(\overline{x}_{j}, \overline{v}_{j}\right)$, where $\overline{z}_{j} = \left(\overline{x}_{j}, \overline{v}_{j}\right)$ for $j = n,n +1$ for $n \in \mathbb{N}$. We have the following update rule for $\overline{z}_{n}$
\begin{align*}
    \overline{x}_{n+1} &= \overline{x}_{n} + h\left(\overline{v}_{n} - hQ \overline{x}_{n}\right), \quad \overline{v}_{n+1} = \eta(\overline{v}_{n} - hQ\overline{x}_{n}),
\end{align*}
where by mean value theorem we define $Q = \int^{1}_{t = 0}\nabla^{2}U(\Tilde{x}_{n} + t(x_{n} - \Tilde{x}_{n}))dt$, then $\nabla U\left(\Tilde{x}_{n}\right) - \nabla U\left(x_{n}\right) = Q\overline{x}_{n}$.
In the notation of (\ref{eq:contraction_matrix}) we have that 
\begin{align*}
    A &= -c\left(h\right)I_{d} + 2\left(b\eta + h\right)hQ - \left(a\eta^2+
 2b\eta h + h^{2}\right)h^{2} Q^{2},\\
    B &= \left(b\left(1-\eta\right) - h -bc\left(h\right)\right)I_{d} + \left(a \eta^{2} + 2b\eta h+ h^{2}\right)hQ,\\
    C &= \left(a\left(1 - \eta^2\right) - 2 b \eta h - h^2 - a  c\left(h\right)\right)I_{d} ,
\end{align*}
where $\eta = \exp{\{-\gamma h\}}$. For our choice of $a$ and $b$, $B$ simplifies to 
$
B = -bc\left(h\right) + \left(a \eta^{2} + 2b\eta h+ h^{2}\right)hQ.
$
This simplification motivates our choice of $b$ as we would like to factor out $m$ from $B$ in proving $AC - B^{2} \succ 0$ in a later calculation and this factorization is necessary to remove the dependence on $m$ in our stepsize estimates.

We will now apply Proposition \ref{Prop:PD}. First considering $A$, we have that the eigenvalues for our choice of $a$, $b$ and $c\left(h\right)$ are precisely 
\[
P_{A}\left(\lambda\right) = -c\left(h\right) + 2b\lambda h - \left(a \eta^2  + 2b \eta h + h^2\right)h^{2}\lambda^2 > h\lambda\left(\frac{7b}{4} - \left(h + \frac{b}{2} + \frac{h}{4} \right)\right) > 0,
\]
where $\lambda$ are the eigenvalues of $Q$ and we have used that $h < 1/(\sqrt{4M})$ and $b \geq h$. Hence we have that $A \succ 0$. Now it remains to prove that $AC - B^{2} \succ 0$, where $AC - B^{2}$ is a polynomial of $Q$, which we denote $P_{AC-B^{2}}(Q)$ and hence has eigenvalues dictated by the eigenvalues $\lambda$ of $Q$. 
We have that 
\begin{align*}
    \frac{P(\lambda)}{h\lambda} &= \frac{\left(\eta^{2}-1\right)c\left(h\right)}{hM\lambda} + h\left( \frac{2\left(1 - \eta^{2}\right)}{\left(1 - \eta\right)M} + \frac{c(h)^{2}}{h^{2}M\lambda} - \frac{\eta^{2}\lambda}{M^{2}} \right) \\
    &+ \frac{c(h)}{h}h^{2}\left(\frac{1}{\lambda} + \frac{2}{(1-\eta)}(\frac{\eta}{\lambda} + \frac{\eta^{2} - 1}{M}) + \frac{\eta^{2} \lambda}{M^{2}}\right)\\
    &+ h^{3}\left( -\frac{4\eta}{(1 -\eta)^2} - \frac{2}{(1 -\eta)} - \frac{c(h)^{2}}{
 (1 -\eta)^2 h^2\lambda} - \frac{\lambda}{M} - \frac{2\eta \lambda}{(1 -\eta) M}\right) \\
    &+ \frac{c(h)}{h}h^{4}  \left(\frac{4\eta}{(1 - \eta)^2} + 
    \frac{2}{(1 - \eta)} + \frac{\lambda}{M} + \frac{
  2\eta \lambda}{(1 - \eta)M}\right)\\
    &> \left( \frac{7h\left(1 + \eta \right)}{4M} - \frac{\eta^{2}h}{M} -\frac{h^{2}}{\sqrt{6M}} \right) + h^{3}\left( -\frac{1}{(1 - \eta)^{2}}\left(4 + \frac{1}{96}\right) -1 - \frac{4}{\left(1-\eta\right)} \right)\\
      &> h\left(\frac{7}{4M} - \frac{1}{6M} - \frac{2}{3M}\left(1 + \frac{1}{384} \right) - \frac{1}{6M}  - \frac{2}{3M} \right) > 0,
\end{align*}
where we have used the fact that $h < \frac{1 - \eta}{\sqrt{6M}}$.
This holds for any $\lambda \in [m,M]$ and hence $AC - B^{2} \succ 0$ and our contraction results hold.
\end{proof}


\begin{proof}[Proof of Theorem \ref{Theorem:OAB}]
We will denote two synchronous realisations of OAB as $\left(x_{j},v_{j}\right)$ and $\left(\Tilde{x}_{j}, \Tilde{v}_{j}\right)$ for $j \in \mathbb{N}$. Now we will denote $\overline{x}_{j} := \left(\Tilde{x}_{j} - x_{j}\right)$, $\overline{v}_{j} = \left(\Tilde{v}_{j} - v_{j}\right)$ and $\overline{z}_{j} = \left(\overline{x}_{j}, \overline{v}_{j}\right)$, where $\overline{z}_{j} = \left(\overline{x}_{j}, \overline{v}_{j}\right)$ for $j = n,n +1$ for $n \in \mathbb{N}$. We have the following update rule for $\overline{z}_{n}$
\begin{align*}
    \overline{x}_{n+1} = \overline{x}_{n} + h \eta \overline{v}_{n}, \quad\overline{v}_{n+1} = \eta\overline{v}_{n} - hQ\overline{x}_{n+1},
\end{align*}
where we define $Q= \int^{1}_{t = 0}\nabla^{2}U(\Tilde{x}_{n+1} + t(x_{n+1} - \Tilde{x}_{n+1}))dt$ by mean value theorem,
then $\nabla U\left(\Tilde{x}_{n+1}\right) - \nabla U\left(x_{n+1}\right) = Q\overline{x}_{n+1}$. 
In the notation of (\ref{eq:contraction_matrix}) we have that 
\begin{align*}
    A &= -c\left(h\right)I_{d}  + 2 b h Q - a h^2 Q^2,\\
    B &= \left(b\left(1-\eta\right)  - \eta h   - bc\left(h\right)\right)I_{d}  + (a \eta h + 2 b \eta h^2) Q - 
 a \eta h^3 Q^2,\\
    C &= \left(a\left(1-\eta^2\right) - 2 b \eta^2 h - \eta^2 h^2 - ac\left(h\right)\right)I_{d} + (2 a \eta^2 h^2 + 2 b \eta^2 h^3) Q - a \eta^2 h^4 Q^2,
\end{align*}
where $\eta = \exp{\{-\gamma h\}}$. For our choice of $a$ and $b$, $B$ simplifies to 
$
     B =  -b c\left(h\right) + (a \eta h + 2 b \eta h^2) Q - 
 a \eta h^3 Q^2.
$
We will now apply Proposition \ref{Prop:PD}, first considering $A$ we have that the eigenvalues are precisely
 \begin{align*}
     P_{A}(\lambda) = h\lambda\left(-\frac{c\left(h\right)}{h\lambda} + 2 b - a h \lambda\right) &\geq h\lambda\left(\frac{7b}{4} - h\right) > 0,
 \end{align*}
Hence we have that $A \succ 0$ and we note that this condition enforces a dependency of $h$ on $\gamma$ to ensure convergence. Now it remains to prove that $AC - B^{2} \succ 0$, now we have that $AC - B^{2}$ is a polynomial of $Q$, which we denote $P_{AC - B^{2}}(Q)$ and hence has eigenvalues dictated by the eigenvalues $\lambda$ of $Q$. 

Using the fact that $h\gamma \geq 1-\eta \geq \frac{h\gamma}{2}$ and hence $\eta \geq \frac{3}{4}$ for $h < \frac{1}{4\gamma}$. Also using that $h \leq \frac{1-\eta}{\sqrt{6M}}$ we have that
\begin{align*}
    &\frac{P_{AC - B^{2}}\left(\lambda\right)}{h\lambda} =  \frac{\left(\eta^2 -1\right) c\left(h\right)}{h M \lambda} + h \left(\frac{2 \eta\left(1 - \eta^{2}\right)}{\left(1 - \eta\right) M} + \frac{c\left(h\right)^{2}}{h^{2} M \lambda} - \frac{\lambda}{M^{2}}\right)\\
    &+ h^{2} \frac{c\left(h\right)}{h} \left( - \frac{2\eta}{ \left(1 - \eta\right) M} - \frac{2\eta^2}{M} + \frac{2\eta^2}{\left(1 - \eta\right)M} + \frac{\eta^{2}}{\lambda} + \frac{2\eta^{3}}{\left(1 - \eta\right)\lambda} + \frac{ \lambda}{M^2} \right)\\
    &+ h^3 \left(-\frac{4\eta^{4}}{\left(1 - \eta\right)^{2}} -\frac{2\eta^{3}}{1 - \eta} - \frac{\eta^{2}c\left(h\right)^{2}}{\left(1-\eta\right)^{2}h^{2}\lambda} + \frac{\lambda\eta^{2}}{M}  + \frac{2\eta^{3}\lambda}{\left(1-\eta \right)M}\right)\\
 &+h^{4}\frac{c\left(h\right)}{h}\left(\frac{4\eta^{3}}{\left(1 - \eta \right)^{2}} - \frac{2\eta^{3}}{\left(1- \eta \right)} + \frac{\eta^{2}\lambda}{M} - \frac{2\eta^{2}\lambda}{\left( 1 -\eta\right)M}\right)\\
 &\geq \frac{5h}{4M} + \left( -\frac{h}{6M} \right) + h^3 \left(- \frac{4}{\left( 1 -\eta \right)^{2}} - \frac{1}{96 \left(1 - \eta \right)^{2}} 
 - \frac{2}{\left(1 - \eta \right)}\right) - \frac{h^{3}}{12} > 0.
\end{align*}
Hence $AC - B^{2} \succ 0$ and our contraction results hold.
\end{proof}


\begin{proof}[Proof of Theorem \ref{Theorem:BAOAB}]
We first note that $\left(\mathcal{BAOAB}\right)^{n} = \mathcal{BAO}\left(\mathcal{ABAO}\right)^{n-1}\mathcal{AB}$. We will now focus our attention on proving contraction of $\mathcal{ABAO}$, by doing this we only have to deal with a single evaluation of the Hessian at each step. We will denote two synchronous realizations of ABAO as $\left(x_{j},v_{j}\right)$ and $\left(\Tilde{x}_{j}, \Tilde{v}_{j}\right)$ for $j \in \mathbb{N}$.
Now we will denote $\overline{x}_{j} := \left(\Tilde{x}_{j} - x_{j}\right)$, $\overline{v}_{j} = \left(\Tilde{v}_{j} - v_{j}\right)$ and $\overline{z}_{j} = \left(\overline{x}_{j}, \overline{v}_{j}\right)$, where $\overline{z}_{j} = \left(\overline{x}_{j}, \overline{v}_{j}\right)$ for $j = n,n +1$ for $n \in \mathbb{N}$. We have the following update rule for $\overline{z}_{n}$
\begin{align*}
    \overline{x}_{n+1} &= \overline{x}_{n} + h \overline{v}_{n} - \frac{h^{2}}{2}Q\left(\overline{x} + \frac{h}{2}\overline{v}\right),\quad \overline{v}_{n+1} = \eta\overline{v}_{n} - h\eta Q\left(\overline{x} + \frac{h}{2}\overline{v}\right),
\end{align*}
where by mean value theorem we define $Q = \int^{1}_{t = 0}\nabla^{2}U(\Tilde{x}_{n} + \frac{h}{2}\Tilde{v}_{n} + t(x_{n} - \Tilde{x}_{n} + \frac{h}{2}(v_{n} - \Tilde{v}_{n})))dt$,
then $\nabla U\left(\Tilde{x}_{n} + \frac{h}{2}\Tilde{v}_{n}\right) - \nabla U\left(x_{n} + \frac{h}{2}v_{n}\right) = Q \left(\overline{x} + \frac{h}{2}\overline{v}\right)$.
In the notation of (\ref{eq:contraction_matrix}) we have that for this scheme
\begin{align*}
    A &= -c\left(h\right)I_{d} + \left(2 b \eta h + h^2\right) Q + \left(-a \eta^2 h^2 - b \eta h^3 - 
    \frac{1}{4} h^4\right) Q^2,\\
    B &= \left(b\left(1 - \eta\right) - h - 
 b c\left(h\right)\right)I_{d} + \left(a \eta^2 h + 2 b \eta h^2 + 
    \frac{3}{4} h^3\right) Q\\
    &+ \left(-\frac{1}{2} a \eta^2 h^3 - \frac{1}{2} b \eta h^4 - \frac{1}{8} h^5\right) Q^2,\\
    C &= \left(a\left(1 - \eta^{2}\right) - 2 b \eta h - h^2 - 
 a c\left(h\right)\right)I_{d} + \left(a \eta^2 h^2 + \frac{3}{2} b \eta h^3 + 
    \frac{1}{2} h^4\right) Q \\
    &+ \left(-\frac{1}{4} a \eta^2 h^4 - \frac{1}{4} b \eta h^5 - \frac{1}{16} h^6\right) Q^2,
\end{align*}
where $\eta = \exp{\{-\gamma h\}}$. For our choice of $a$ and $b$, $B$ simplifies to 
$
     B = -b c\left(h\right) + \left(a \eta^2 h + 2 b \eta h^2 + 
    \frac{3}{4} h^3\right) Q + \left(-\frac{1}{2} a \eta^2 h^3 - \frac{1}{2} b \eta h^4 - \frac{1}{8} h^5\right) Q^2.
$
Now it is sufficient to prove that $A \succ 0$ and that $C - BA^{-1}B \succ 0$, noting that $A$, $B$ and $C$ commute as they are all polynomials in $Q$; it is sufficient to prove that $A \succ 0$ and $AC - B^{2} \succ 0$. First considering $A$ we have 
 that $A$ is symmetric and hence it is positive definite if and only if all its eigenvalues are positive. We note that the eigenvalues of $A$ are precisely
\begin{align*}
     P_{A}\left(\lambda\right) &= h \left( - \frac{c\left(h\right)}{h} + (2 b \eta  + h) \lambda + (-a \eta^2 h - b \eta h^2 - 
    \frac{1}{4} h^3) \lambda^2 \right)\\
     &\geq h\lambda\left(\frac{7b\eta}{4} +\frac{3h}{4} - \eta^{2}h  - b\eta h^{2}M - \frac{h^{3}M}{4}\right)\geq h\lambda\left(\frac{b\eta}{2} + \frac{3h}{4} - \frac{h}{16}\right) > 0,
 \end{align*}
 where $\lambda$ are the eigenvalues of $Q$, where $m \leq \lambda \leq M$ and we have used the fact that $h < \frac{1 - \eta}{\sqrt{4M}}$ and $b\geq h$. 
Hence we have that $A \succ 0$. Now it remains to prove that $AC - B^{2} \succ 0$, now we have that $AC - B^{2}$ is a polynomial of $Q$, which we denote $P_{AC - B^{2}}(Q)$ and hence has eigenvalues dictated by the eigenvalues $\lambda$ of $Q$. Now considering


\begin{align*}
    &\frac{P_{AC - B^{2}}\left(\lambda\right)}{h\lambda} =  \frac{\left(\eta^2 -1\right) c\left(h\right)}{h M \lambda} + h \left( \frac{1 - \eta^{2}}{M} -\frac{\eta^{2}\lambda}{M^{2}}+ \frac{2 \eta\left(1 - \eta^{2}\right)}{\left(1 - \eta\right) M} + \frac{c\left(h\right)^{2}}{h^{2} M \lambda} \right)\\
    &+ h^{2}\frac{c\left(h\right)}{h} \left( - \frac{\left( 1 + \eta \right)^{2}}{M} + \frac{1 + \eta}{\left(1 - \eta \right)\lambda} + \frac{\eta^{2}\lambda}{M^{2}}\right)\\
    &+ h^3 \left(-1 - \frac{4\eta}{\left(1 - \eta\right)^{2}} - \frac{c\left(h\right)^{2}}{h^{2}\left(1 - \eta\right)^{2}\lambda} - \frac{\lambda}{4M} + \frac{3\eta^{2}\lambda}{4M} -\frac{\eta\lambda\left(1 - \eta^{2}\right)}{\left(1-\eta\right)M}\right)\\
 &+ h^{4}\frac{c\left(h\right)}{h}\left(1 + \frac{4\eta}{\left(1 - \eta\right)^{2}} + \frac{\lambda\left(1 +\eta^{2}\right)}{4M} + \frac{\eta\lambda}{M} \right) + h^{5}\lambda\left(\frac{3}{16} + \frac{\eta}{\left(1 - \eta\right)^{2}}\right)\\
 &+ h^{6}\frac{c\left(h\right)}{h}\left(-\frac{3\lambda}{16} - \frac{\eta\lambda}{\left(1 - \eta\right)^{2}}\right) \geq -\frac{\left(1 +\eta\right)h}{4M} + h\left(\frac{1 + 2\eta}{M}\right)\\
 &+ h^3 \left(-\frac{5}{4} - \frac{4\eta}{\left( 1 -\eta \right)^{2}} - \frac{1}{64 \left(1 - \eta\right)^{2}} - \eta\left(1 + \eta\right)\right) - \frac{h^{3}}{32} > 0, 
\end{align*}
where we have used the fact that $h < \frac{1- \eta}{\sqrt{4M}}$.
Hence $AC - B^{2} \succ 0$ and our contraction results hold. All computations can be checked using symbolic computing. We can bound the $\mathcal{AB}$ operator on $||\cdot ||_{a,b}$ by
\begin{align*}
    &||\Psi_{\rm{AB}}\left(\Tilde{x}_{n},\Tilde{v}_{n},h/2\right) - \Psi_{\rm{AB}}\left(x_{n},v_{n},h/2\right)||^{2}_{a,b} \leq \\
    &3\left(\left(1 + \frac{ah^{2}M^{2}}{2}\right) ||\overline{x}_{n}||^{2} + \left(a + \frac{h^{2}}{4} + \frac{ah^{4}M^{2}}{8}\right)||\overline{v}_{n}||^{2} \right) \leq 7||\overline{x}_{n},\overline{v}_{n}||^{2}_{a,b},
\end{align*}
where we have used the norm equivalence in Section \ref{Sec:Quadratic_Norm}. We can also bound 
\begin{align*}
    &||\Psi_{\rm{O}}\left(\Psi_{\rm{BA}}(\Tilde{x}_{n},\Tilde{v}_{n},h/2),h)\right) - \Psi_{\rm{O}}\left(\Psi_{\rm{BA}}(x_{n},v_{n},h/2),h\right)||^{2}_{a,b}\\
    &\leq 3\left(\left(1 + \frac{ah^{2}M^{2}}{4} + \frac{h^{4}M^{2}}{8} \right) ||\overline{x}_{n}||^{2} + \left(\frac{h^{2}}{2} + a \right)||\overline{v}_{n}||^{2} \right)\leq 7 ||\overline{x}_{n},\overline{v}_{n}||^{2}_{a,b}.
\end{align*}
Combining these estimates we have the required result.
\end{proof}


\begin{proof}[Proof of Theorem \ref{Theorem:OBABO}]

We first note that $\left(\mathcal{OBABO}\right)^{n} = \mathcal{OB}\left(\mathcal{ABOB}\right)^{n-1}\mathcal{ABO}$. We will now focus our attention on proving contraction of $\mathcal{ABOB}$. Note we only have to deal with a single evaluation of the Hessian at each step as the position variable is not updated between gradient evaluations. We will denote two synchronous realisations of ABOB as $\left(x_{j},v_{j}\right)$ and $\left(\Tilde{x}_{j}, \Tilde{v}_{j}\right)$ for $j \in \mathbb{N}$.
Now we will denote $\overline{x}_{j} := \left(\Tilde{x}_{j} - x_{j}\right)$, $\overline{v}_{j} = \left(\Tilde{v}_{j} - v_{j}\right)$ and $\overline{z}_{j} = \left(\overline{x}_{j}, \overline{v}_{j}\right)$, where $\overline{z}_{j} = \left(\overline{x}_{j}, \overline{v}_{j}\right)$ for $j = n,n +1$ for $n \in \mathbb{N}$. We have the following update rule for $\overline{z}_{n}$
\begin{align*}
    \overline{x}_{n+1} &= \overline{x}_{n} + h \overline{v}_{n}, \quad \overline{v}_{n+1} = \eta\overline{v}_{n} - \frac{h}{2}(\eta + 1) Q\left(\overline{x} + h\overline{v}\right),
\end{align*}
where by mean value theorem we define $Q = \int^{1}_{t = 0}\nabla^{2}U(\Tilde{x}_{n} + h\Tilde{v}_{n} + t(x_{n} - \Tilde{x}_{n} + h(v_{n} - \Tilde{v}_{n})))dt$, then $\nabla U\left(\Tilde{x}_{n} + h\Tilde{v}_{n}\right) - \nabla U\left(x_{n} + hv_{n}\right) = Q \left(\overline{x} + h\overline{v}\right)$. In the notation of (\ref{eq:contraction_matrix}) we have that for this scheme
\begin{align*}
    A &= -c\left(h\right)I_{d} + bh\left(1 + \eta\right) Q - (1 + \eta)^{2}\frac{ah^{2}Q^{2}}{4},\\
    B &= \left(b\left(1 - \eta\right) - h - 
 b c\left(h\right)\right)I_{d} + \left(\frac{1}{2}a\eta + bh\right) \left(\eta + 1\right)hQ - \left(\eta +1\right)^{2}\frac{ah^{3}}{4} Q^2,\\
    C &= \left(a\left(1 - \eta^{2}\right) - 2 b \eta h - h^2 - 
 a c\left(h\right)\right)I_{d} + \left(a\eta + bh\right)\left(\eta + 1\right) h^{2}Q - a\left(\eta + 1 \right)^{2}\frac{h^{4}}{4} Q^2,
\end{align*}
where $\eta = \exp{\{-\gamma h\}}$. This form motivates the choice $b = \frac{h}{1 -\eta}$ and $a = \frac{1}{M}$ inspired by the continuous dynamics. For our choice of $a$ and $b$, $B$ simplifies to 
$
     B = - 
 b c\left(h\right) + (\frac{1}{2}a\eta + bh) (\eta + 1)hQ - (\eta +1)^{2}\frac{ah^{3}}{4} Q^2.
$
We will now apply Proposition \ref{Prop:PD}, first considering $A$ we have that the eigenvalues are precisely 
 \begin{align*}
     P_{A}\left(\lambda\right) &= -c\left(h\right) + bh(1 + \eta) \lambda - (1 + \eta)^{2}\frac{ah^{2}\lambda^{2}}{4} \geq h\lambda \left(\frac{3b}{4} + b \eta - \frac{h}{4}  - \frac{3b\eta}{4}\right)>0,
 \end{align*}
 where $\lambda$ are the eigenvalues of $Q$, where $m \leq \lambda \leq M$ and we have used the fact that $b \geq h$.
 Hence we have that $A \succ 0$. Now it remains to prove that $AC - B^{2} \succ 0$, now we have that $AC - B^{2}$ is a polynomial of $Q$, which we denote $P_{AC - B^{2}}(Q)$ and hence has eigenvalues dictated by the eigenvalues $\lambda$ of $Q$. Now considering 
\begin{align*}
    &\frac{P_{AC - B^{2}}\left(\lambda\right)}{h\lambda} =  \frac{\left(\eta^2 -1\right) c\left(h\right)}{h M \lambda} + h \left( \frac{\left(1 + \eta\right)^{2}}{M} + \frac{c\left(h\right)^{2}}{h^2 M \lambda} - \frac{\left(1 + \eta\right)^{2}\lambda}{4M^{2}}\right)\\
    &+ h^{2}\frac{c\left(h\right)}{h} \left( -\frac{\left(1 + \eta\right)^{2}}{M} - \frac{1}{\lambda} + \frac{2}{\left(1 - \eta \right)\lambda} + \frac{\lambda\left(1 + \eta \right)^{2}}{4 M^{2}}\right)\\
    &+  h^3 \left(-\frac{\left(1 + \eta\right)^{2}}{\left(1 - \eta\right)^{2}} - \frac{c\left(h\right)^{2}}{\left(1 - \eta\right)^{2}h^{2}\lambda} - \frac{\lambda \left(1 + \eta\right)^{2}}{4M} + \frac{\lambda\left(1 + \eta\right)^{2}}{2\left(1-\eta\right)M}\right)\\
 &+h^{4}\frac{c\left(h\right)}{h}\left(\frac{\left(1 +\eta\right)^{2}}{\left(1 - \eta \right)^{2}} + \frac{\lambda\left(1 +\eta\right)^{2}}{4M} - \frac{\lambda \left(1 + \eta\right)^{2}}{2M\left(1-\eta\right)}\right)\\
 &> -\frac{h\left(1 + \eta\right)}{4M} + h \left(\frac{3\left(1 + \eta\right)^{2}}{4M}\right) - h\left(\frac{3 \left(1 + \eta\right)^{2}}{64M} \right) + h\left(-\frac{\left(1 + \eta\right)^{2}}{4M} -\frac{1}{64M} \right) >0,
\end{align*}
where we have used the fact that $h < \frac{1 -\eta}{\sqrt{4M}}$. Hence $AC - B^{2} \succ 0$ and our contraction results hold. All computations can be checked using symbolic computing. We can bound the $\mathcal{ABO}$ operator on $||\cdot ||_{a,b}$ by 
\begin{align*}
    &||\Psi_{\rm{BO}}\left(\Psi_{\rm{A}}\left(\Tilde{x}_{n},\Tilde{v}_{n},h\right),h/2\right) - \Psi_{\rm{BO}}\left(\Psi_{\rm{A}}\left(x_{n},v_{n},h\right),h/2\right)||^{2}_{a,b}
    \leq 8 ||\overline{x}_{n},\overline{v}_{n}||^{2}_{a,b},
\end{align*}
where we have used the norm equivalence in Section \ref{Sec:Quadratic_Norm}. Similarly, we can also bound
\begin{align*}
    ||\Psi_{\rm{OB}}\left(\Tilde{x}_{n},\Tilde{v}_{n},h/2\right) - \Psi_{\rm{OB}}\left(x_{n},v_{n},h/2\right)||^{2}_{a,b} &
    \leq 6 ||\overline{x}_{n},\overline{v}_{n}||^{2}_{a,b}.
\end{align*}
Combining these estimates we have the required result.
\end{proof}


\begin{proof}[Proof of Theorem \ref{Theorem:SES}]
We remark that synchronous coupling between two realisations of the stochastic Euler scheme results in a synchronous coupling of $\left(\zeta_{n},\omega_{n}\right)_{n \in \mathbb{N}}$. Now we will denote $\overline{x}_{j} := \left(\Tilde{x}_{j} - x_{j}\right)$, $\overline{v}_{j} = \left(\Tilde{v}_{j} - v_{j}\right)$ and $\overline{z}_{j} = \left(\overline{x}_{j}, \overline{v}_{j}\right)$, where $\overline{z}_{j} = \left(\overline{x}_{j}, \overline{v}_{j}\right)$ for $j = n,n +1$ for $n \in \mathbb{N}$. We have the following update rule for $\overline{z}_{n}$
\begin{align*}
    \overline{x}_{n+1} = \overline{x}_{n} + \frac{1 - \eta}{\gamma}\overline{v}_{n} - \frac{\gamma h + \eta - 1}{\gamma^{2}}Q\overline{x}_{n}, \quad \overline{v}_{n+1} = \eta \overline{v}_{n} - \frac{1 - \eta}{\gamma}Q\overline{x}_{n},
\end{align*}
where by mean value theorem we define $Q = \int^{1}_{t = 0}\nabla^{2}U(\Tilde{x}_{n}  + t(x_{n} - \Tilde{x}_{n}))dt$, then $\nabla U\left(\Tilde{x}_{n}\right) - \nabla U\left(x_{n}\right) = Q \left(\overline{x} + h\overline{v}\right)$.  In the notation of (\ref{eq:contraction_matrix}) we have that for this scheme
\begin{align*}
    A &= -c\left(h\right)I_{d} + 2\left(\frac{b\left(1 - \eta\right)}{\gamma} + \frac{\eta - 1 + \gamma h}{\gamma^{2}}\right)Q\\
    &- \left(\frac{a\left( 1 - \eta \right)^{2}}{\gamma^{2}} + \frac{2b\left( 1-\eta\right)\left(-1 + \eta + \gamma h \right)}{\gamma^{3}} + \frac{\left(-1 + \eta + \gamma h \right)^{2}}{\gamma^{4}}\right) Q^{2},\\
    B &= \left(b\left(1 - \eta\right) - \frac{\left(1 - \eta\right)}{\gamma} - bc\left(h\right)\right)I_{d}\\
    &+ \left(\frac{a\eta\left(1- \eta\right)}{\gamma} + \frac{b\left(1 - \eta\right)^{2}}{\gamma^{2}} + \frac{b\eta\left(-1 + \eta + \gamma h\right)}{\gamma^{2}}  + \frac{\left(1 - \eta \right)\left(-1 + \eta + \gamma h \right)}{\gamma^{3}}\right)Q,\\
    C &= \left(a (1 - \eta^{2})
 -ac\left(h\right) - \frac{2b\eta \left(1 - \eta\right)}{\gamma} - \frac{\left( 1 - \eta\right)^{2}}{\gamma^{2}}\right)I_{d},
  \end{align*}
where $\eta = \exp{\{-\gamma h\}}$. This form motivates the choice $b = \frac{1}{\gamma}$ and $a = \frac{1}{M}$ inspired by the continuous dynamics. For our choice of $a$ and $b$, $B$ simplifies to $B = -bc\left(h\right) + \mathcal{O}\left(Q\right)$.
We will now apply Proposition \ref{Prop:PD}, first considering $A$ we wish to show that all its eigenvalues are positive which are precisely 
 \begin{align*}
 P_{A}(\lambda) &:= -c\left(h\right) + 2\left(\frac{b\left(1 - \eta\right)}{\gamma} + \frac{\eta - 1 + \gamma h}{\gamma^{2}}\right)\lambda \\
 &- \left(\frac{a\left( 1 - \eta \right)^{2}}{\gamma^{2}} + \frac{2b\left( 1-\eta\right)\left(-1 + \eta + \gamma h \right)}{\gamma^{3}} + \frac{\left(-1 + \eta + \gamma h \right)^{2}}{\gamma^{4}}\right) \lambda^{2},\\
 &\geq h\lambda \left(\frac{7}{4\gamma} - \left(h + \frac{h^{2}M}{\gamma} + \frac{h^{3}M}{4}\right)\right) > 0,
 \end{align*}
 where $\lambda$ are the eigenvalues of $Q$, where $m \leq \lambda \leq M$ and using the fact that, $\gamma^{2} \geq 4M$, $1 - \eta \leq h\gamma$, $h\gamma + \eta - 1 \leq \frac{\left(h\gamma\right)^{2}}{2}$ and $h<\frac{1}{2\gamma}$.
 Hence we have that $A \succ 0$. Now it remains to prove that $AC - B^{2} \succ 0$, now we have that $AC - B^{2}$ is a polynomial of $Q$, which we denote $P_{AC - B^{2}}(Q)$ and hence has eigenvalues dictated by the eigenvalues $\lambda$ of $Q$. Since the terms are more complicated than the previous discretizations we choose a convenient way of expanding the expression which can obtain positive definiteness. That is to expand the expression in terms of $a$. By using symbolic computing one can show that $P_{AC-B^{2}}(\lambda) = c_{0} + c_{1}a + c_{2}a^{2}$, where
\begin{align*}
    &c_{1} + c_{2} a =  \left(\eta^{2} - 1\right) c\left(h\right) + c\left(h\right)^{2} + 2\left(1 - \eta^{2}\right)\left(\frac{b\left(1 - \eta\right)}{\gamma}  + \frac{-1 + \eta + \gamma h}{\gamma^{2}}\right)\lambda \\
    &+ \frac{2b\left(1 -\eta\right)\eta  c\left(h\right) \lambda}{\gamma}  - 2 \left( \frac{b\left(1 - \eta\right)}{\gamma} + \frac{-1 + \eta + \gamma h}{\gamma^{2}}\right)c\left(h\right)\lambda + \frac{\left( 1- \eta\right)^{4}\lambda^{2}}{\gamma^{4}}\\
    &-\frac{2\eta\left(1-\eta\right)^{2}\left(-1 + \eta + \gamma h\right)\lambda^{2}}{\gamma^{4}}  - \frac{2b\left(1 - \eta\right)\left(-1 + \eta + \gamma h \right) \lambda^{2}}{\gamma^{3}} \\
    &- \frac{\left(1 - \eta^{2} \right)\left(-1 + \eta +\gamma h\right)^{2}\lambda^{2}}{\gamma^{4}} + \frac{2b\left(1 - \eta\right)\left(-1 + \eta + \gamma h \right) c\left(h\right) \lambda^{2}}{\gamma^{3}} \\
    &+ \frac{\left(-1 + \eta + \gamma h \right)^{2}c\left(h\right)\lambda^{2}}{\gamma^{4}}+ a\left( -\frac{\left(1 - \eta\right)^{2}\lambda^{2}}{\gamma^{2}} + \frac{\left(1 - \eta\right)^{2}c\left(h\right)\lambda^{2}}{\gamma^{2}}\right)\\
    &\geq c_{1} - \frac{\left(1 - \eta\right)^{2}\lambda}{\gamma^{2}} + \frac{\left(1 - \eta\right)^{2}c\left(h\right)\lambda}{\gamma^{2}} \\
    &\geq \left(\eta^{2} - 1\right) c\left(h\right) + \left(1 - \eta^{2}\right)\left(\frac{2h}{\gamma} - \frac{1 - \eta}{\left(1 + \eta \right)\gamma^{2}}\right)\lambda + ...\\
    &\geq \lambda\left(\left(-\frac{h^{2}}{2}\right) + h^{2}\left(2 - \frac{1}{1 +\eta} \right) + ...\right)\\
    &> \lambda \left(\frac{h^{2}}{2} - \frac{h^2}{16}- \frac{h^{3}\gamma}{16} - \frac{h^{3}\gamma}{8} - \frac{h^{3}\gamma}{16}  \right) \geq \lambda \left(\frac{7h^{2}}{16} - \frac{h^{3}\gamma}{4} \right),
\end{align*}
where $h < \frac{1}{2\gamma}$, $\gamma^{2} \geq 8M \geq 8m$, $\frac{h\gamma}{2}\leq 1 - \eta \leq h\gamma$ and $h\gamma + \eta - 1 \leq \frac{\left(h\gamma\right)^{2}}{2}$. Further, we have that 
\begin{align*}
    &c_{0} = \frac{\left(1 - \eta \right)^{2}c\left(h\right)}{\gamma^{2}} + \frac{2b\left(1 - \eta \right) \eta  c\left(h\right)}{\gamma} - b^{2}c\left(h\right)^{2} - \frac{2b\left(1 - \eta\right)^{3}\lambda}{\gamma^{3}}\\
    &- \frac{2\left(1 - \eta\right)^{2}\left(-1 + \eta + \gamma h \right)\lambda}{\gamma^{4}} - \frac{4b^{2}\eta\left(1- \eta\right)^{2}\lambda}{\gamma^{2}} - \frac{4b\eta\left(1 - \eta\right)\left(-1 + \eta + \gamma h\right)\lambda}{\gamma^{3}}\\
    &-\frac{b^2 \eta ^2 \lambda^2 (\eta +\gamma h-1)^2}{\gamma^4} +\frac{2 b^2 (1-\eta )^2  c\left(h\right)
   \lambda}{\gamma^2} +\frac{2 b^2 \eta   c\left(h\right) \lambda (\eta +\gamma h-1)}{\gamma^2}\\
   &+\frac{2 b^2 \eta 
   (1-\eta )^2 \lambda^2 (\eta +\gamma h-1)}{\gamma^4}-\frac{b^2 (1-\eta )^4
   \lambda^2}{\gamma^4}+\frac{2 b (1-\eta ) c\left(h\right) \lambda (\eta +\gamma h-1)}{\gamma^3}\\
   &> \lambda \left(-\frac{2\left(h\gamma\right)^{3}}{\gamma^{4}} -\frac{\left(h\gamma\right)^{4}}{\gamma^{4}} -\frac{4\left(h\gamma\right)^{2}}{\gamma^{4}} -\frac{2\left(h\gamma\right)^{3}}{\gamma^{4}} - \frac{\left(h\gamma \right)^{4}\lambda}{4\gamma^{6}} - \frac{\left(h\gamma \right)^{4}\lambda}{\gamma^{6}}\right)> \lambda \left(-\frac{7h^{2}}{\gamma^{2}}\right),
\end{align*}
now we can combine this with the previous estimate and we have 
\[P_{AC - B^{2}}(\lambda) > \lambda \left( \frac{7h^{2}}{16M} - \frac{h^{3}\gamma}{4M}  - \frac{7h^{2}}{\gamma^{2}} \right) > h^2\lambda\left(\frac{5}{16M} - \frac{7}{\gamma^{2}}\right) \geq 0,\]
which is true when $\gamma \geq 5\sqrt{M}$. Hence $AC - B^{2} \succ 0$ and our contraction results hold. All computations can be checked using symbolic computing.
\end{proof}

\section{Asymptotic Bias of BAOAB}\label{Appendix:Bias}

\begin{proof}[Proof of Proposition \ref{prop:BAOAB_local_error}]
We start by estimating $\Delta_{x}$. We use the following Taylor expansion for the Hamiltonian dynamics
\begin{align*}
    \Psi_{\textnormal{H}}(z,h) 
    &= \left(x + hv - \int^{h}_{0}\nabla U(x(t))(h-t)dt, v - h\nabla U(x) - \int^{h}_{0}\nabla^{2}U(x(t))v(t)(h-t)dt\right),
\end{align*}
we then define $(\overline{x}(t),\overline{v}(t))$ to be Hamiltonian dynamics initialised at $\Psi_{\textnormal{O}}(\Psi_{\textnormal{H}}(z,h/2),h)$ at time $t > 0$ and $(\overline{x},\overline{v}) := (\overline{x}(0),\overline{v}(0))$. Then the $x$-component of the HOH scheme is 
\begin{align*}
   x_{\textnormal{HOH}}& := x + \frac{h}{2}v(1 + \eta) - \int^{h/2}_{0}\nabla U(x(t))\left(\frac{h}{2}-t\right)dt - \int^{h/2}_{0}\nabla U(\overline{x}(t))\left(\frac{h}{2}-t\right)dt\\
   &+ \frac{h}{2}\left(\eta \left(- \frac{h}{2}\nabla U(x) - \int^{h/2}_{0}\nabla^{2}U(x(t))v(t)\left(\frac{h}{2}-t\right)dt\right) + \sqrt{1-\eta^{2}}\xi\right) 
\end{align*}
and the $x$-component of the BAOAB scheme is
\begin{align*}
   x_{\textnormal{BAOAB}} &:=x' + \frac{h}{2}v'(1+\eta) -\frac{h^{2}}{4}(1+\eta)\nabla U(x') + \frac{h}{2}\sqrt{1-\eta^{2}}\xi. 
\end{align*}
Therefore the difference $\Delta_{x} = x_{\textnormal{HOH}} - x_{\textnormal{BAOAB}}$ in $x$ satisfies
\begin{align*}
   &\|\Delta_{x}\|_{L^{2}} \leq \left(1 + \eta\frac{h^{2}}{4}M\right)\|x - x'\|_{L^{2}} + \frac{h}{2}(1+\eta)\|v-v'\|_{L^{2}}\\
   &+\Bigg\|\int^{h/2}_{0}\left(\nabla U(x(t)) - \nabla U(x')\right)\left(\frac{h}{2}-t\right)dt - \eta\frac{h}{2}\int^{h/2}_{0}\nabla^{2}U(x(t))v(t)\left(\frac{h}{2}-t\right)dt \\
   &+ \int^{h/2}_{0}\left(\nabla U(\overline{x}(t)) - \nabla U(x')\right)\left(\frac{h}{2}-t\right)dt\Bigg\|_{L^{2}}.\\
   \intertext{We now bound the final expression by}
   &  M\int^{h/2}_{0}\left\|x(t) - x'\right\|_{L^{2}}\left(\frac{h}{2} - t\right)dt  + \frac{h\eta M}{2}\int^{h/2}_{0}\|v(t)\|_{L^{2}}\left(\frac{h}{2}-t\right)dt\\
   &+ M\int^{h/2}_{0}\left\|\overline{x}(t) - x'\right\|_{L^{2}}\left(\frac{h}{2} - t\right)dt,
\end{align*}
where we have that the second term is bounded by $\frac{h^{3}\eta M}{16}\sqrt{d}$ and considering the first term we have for $t \in [0,h/2]$
\begin{align*}
    \|x(t)-x'\|_{L^{2}} &\leq \|x - x'\|_{L^{2}} + \|tv - \int^{t}_{0}\nabla U(x(s))(t-s)ds\|_{L^{2}} \leq \|x - x'\|_{L^{2}} +  \frac{3h\sqrt{d}}{4},
\end{align*}
similarly we can bound $M\int^{h/2}_{0}\left\|\overline{x}(t) - x'\right\|_{L^{2}}\left(\frac{h}{2} - t\right)dt$. Therefore we have by summing the estimates
\[
\|\Delta_{x}\|_{L^{2}} \leq \left(1 + (1+\eta)\frac{h^{2}}{4}M\right)\|x - x'\|_{L^{2}} + \frac{h}{2}(1+\eta)\|v-v'\|_{L^{2}} + \frac{3h^{3}M\sqrt{d}}{8},
\]
where we have used the fact that $h < \frac{1}{2\sqrt{M}}$.
Now we estimate $\Delta_{v}$, considering the velocity components we have that the $v$-component of the HOH scheme is
\begin{align*}
   v_{\textnormal{HOH}}&:=\eta\left(v-\frac{h}{2}\nabla U(x) - \int^{h/2}_{0}\nabla^{2}U(x(t))v(t)\left(\frac{h}{2} - t\right)dt\right) + \sqrt{1-\eta^{2}}\xi \\
   &- \frac{h}{2}\nabla U(\overline{x}) - \int^{h/2}_{0}\nabla^{2}U(\overline{x}(t))\overline{v}(t)\left(\frac{h}{2} - t\right)dt, 
\end{align*}
where
$
\overline{x}:= x + \frac{h}{2}v - \int^{h/2}_{0}\nabla U(x(t))\left(\frac{h}{2} - t\right)dt.
$
The $v$-component of BAOAB is
\[
v_{\textnormal{BAOAB}} := \eta (v'- \frac{h}{2}\nabla U(x')) + \sqrt{1-\eta^{2}}\xi - \frac{h}{2}\nabla U(\hat{x}),
\]
where
$
    \hat{x}:= x' + \frac{h}{2}(1+\eta)v' - \frac{h^{2}}{4}(1+\eta)\nabla U(x') + \frac{h}{2}\sqrt{1-\eta^{2}}\xi.
$
For $\Delta_{v} = v_{\textnormal{HOH}} - v_{\textnormal{BAOAB}}$ we have that 
\begin{align*}
    &\Delta_{v} = \eta(v - v') - \frac{h\eta}{2}\left(\nabla U(x) - \nabla U(x')\right)  -\eta\int^{h/2}_{0}\nabla^{2}U(x(t))v(t)\left(\frac{h}{2} - t\right)dt\\
    &-\frac{h}{2}\left(\nabla U(\overline{x}) - \nabla U(\hat{x})\right) - \int^{h/2}_{0}\nabla^{2}U(\overline{x}(t))\overline{v}(t)\left(\frac{h}{2} - t\right)dt,
\end{align*}
we now consider the Taylor expansion 
\begin{align*}
    &\nabla U(\hat{x}) = \nabla U(\hat{x}_{c})\\
    & -\nabla^{2}U([\hat{x}_{c},\hat{x}]) \left(x-x' + \frac{h}{2}(1+\eta)(v-v') - \frac{h^{2}}{4}(1+\eta)(\nabla U(x) -\nabla U(x'))\right),
\end{align*}
where we define 
\begin{align*}
    \hat{x}_{c} &:= x + \frac{h}{2}(1+\eta)v-\frac{h^{2}}{4}(1+\eta)\nabla U(x) + \frac{h}{2}\sqrt{1-\eta^{2}}\xi,\\
    \nabla^{2}U([v_{1},v_{2}]) &:= \int^{1}_{0}\nabla^{2}U(v_{1} + s(v_{2}-v_{1}))ds,
\end{align*}
for any $v_{1},v_{2} \in \mathbb{R}^{d}$.
We then define 
\begin{align*}
    \alpha_{v} &:=  \eta(v - v') - \frac{h\eta}{2}\left(\nabla U(x) - \nabla U(x')\right)\\
    &-\frac{h}{2}\nabla^{2}U([\hat{x}_{c},\hat{x}]) \left(x-x' + \frac{h}{2}(1+\eta)(v-v') - \frac{h^{2}}{4}(1+\eta)\left(\nabla U(x) -\nabla U(x')\right)\right)\\
    &=(\eta I_{d} - \frac{h^{2}(1+\eta)}{4}Q_{2})(v-v') +\left(-\frac{h\eta}{2}Q_{1} - \frac{h}{2}Q_{2} + \frac{h^{3}(1+\eta)}{8}Q_{2}Q_{1}\right)(x-x'),
\end{align*}
where $Q_{1} = \nabla^{2}U([x,x'])$ and $Q_{2} = \nabla^{2}U([\hat{x},\hat{x}_{c}])$. Consider $\Delta_{v} - \alpha_{v}$, which can be written as
\begin{align*}
    &-\eta\int^{h/2}_{0}\nabla^{2}U(x(t))v(t)\left(\frac{h}{2} - t\right)dt +\frac{h}{2}\nabla^{2}U([\hat{x}_{c},\overline{x}])(\hat{x}_{c}-\overline{x}) \\
    &- \int^{h/2}_{0}\nabla^{2}U(\overline{x}(t))\overline{v}(t)\left(\frac{h}{2} - t\right)dt,
\end{align*}
which only contains terms from the continuous dynamics. Removing some third order terms which we can bound in $L^{2}$ by $h^{3}M^{3/2}\sqrt{d}$, where we have used that $h<\frac{1}{2\sqrt{M}}$, we have the additional terms are given by 
\begin{align*}
    &-\eta\int^{h/2}_{0}\nabla^{2}U(x(t))v\left(\frac{h}{2} - t\right)dt +\frac{h}{2}\nabla^{2}U([\hat{x}_{c},\overline{x}])\left(\frac{h}{2}\eta v + \frac{h}{2}\sqrt{1-\eta^{2}}\xi\right)\\
    &- \int^{h/2}_{0}\nabla^{2}U(\overline{x}(t))\left(\eta\Tilde{v} + \sqrt{1-\eta^{2}}\xi\right)\left(\frac{h}{2} - t\right)dt,
\end{align*}
where $\Tilde{v}:= v - \frac{h}{2}\nabla U(x) -\int^{h/2}_{0}\nabla^{2}U(x(t))v(t)\left(\frac{h}{2} - t\right)dt$ and we can bound in $L^{2}$ 
\begin{align*}
    &\Bigg\|\eta\int^{h/2}_{0}\nabla^{2}U(x(t))v\left(\frac{h}{2} - t\right)dt - \eta\frac{h^{2}}{4} \nabla^{2}U([\hat{x}_{c},\overline{x}]) v + \eta\int^{h/2}_{0}\nabla^{2}U(\overline{x}(t)) \Tilde{v} \left(\frac{h}{2} - t\right) dt\Bigg\|_{L^{2}}\\
    &\leq  \frac{\eta h^{3}M^{3/2}\sqrt{d}}{8} + \eta\Bigg(\left\|\int^{h/2}_{0}\left(\nabla^{2}U(x(t)) - \nabla^{2}U([\hat{x}_{c},\overline{x}])\right)v \left(\frac{h}{2} - t\right) dt\right\|_{L^{2}} \\
    &+ \left\|\int^{h/2}_{0}\left(\nabla^{2}U(\overline{x}(t)) - \nabla^{2}U([\hat{x}_{c},\overline{x}])\right)v\left(\frac{h}{2} - t\right)dt\right\|_{L^{2}}\Bigg),
\end{align*}
and you can bound the second term under Assumption \ref{assum:H_Lipschitz} by
\begin{align*}
    &\eta M_{1}\int^{h/2}_{0}\int^{1}_{0}\left\|\left\|x(t)-\hat{x}_{c} - s(\overline{x}-\hat{x}_{c})\right\|\left\|v\right\|\right\|_{L^{2}}\left(\frac{h}{2} - t\right)dsdt\\
    &\leq \eta M_{1}\int^{h/2}_{0}\left\|\left(2h\|v\|+\frac{3}{2}h^{2}\sqrt{Md} + h \sqrt{1-\eta^{2}}\|\xi\|\right)\left\|v\right\|\right\|_{L^{2}}\left(\frac{h}{2} - t\right)dt \leq \frac{\eta h^{3}M_{1}d}{2}
\end{align*}
and similarly for the third term. Without Assumption \ref{assum:H_Lipschitz}, we can bound each of these terms by $\eta \frac{h^{2}M\sqrt{d}}{4}$.
The remaining terms we have are 
\begin{equation}\label{eq:noise_terms}
    \sqrt{1-\eta^{2}}\frac{h^{2}}{4}\nabla^{2}U([\hat{x}_{c},\overline{x}])\xi -  \sqrt{1-\eta^{2}}\int^{h/2}_{0}\nabla^{2}U(\overline{x}(t))\left(\frac{h}{2} - t\right)\xi dt,
\end{equation}
which can be bounded by $3h^{2}M\sqrt{d}/8$ in $L^{2}$. For higher order estimates we have that the $\nabla^{2}U$ terms contain noise and hence we use the Taylor expansions 
\begin{align*}
    &\nabla^{2}U(\overline{x}(t)) = \nabla^{2}U(\overline{x}) +\nabla^{3}U([\overline{x}(t),\overline{x}]) \left(t\overline{v} - \int^{t}_{0}\nabla U(\overline{x}(s))(t-s)ds\right), \\
    &\nabla^{2}U([\hat{x}_{c},\overline{x}]) = \int^{1}_{0}\nabla^{2} U(\hat{x} + s(\overline{x} - \hat{x}) - (1-s)\frac{h}{2}\sqrt{1-\eta^{2}}\xi)ds \\
    &+ \frac{h}{2}\sqrt{1-\eta^{2}}\int^{1}_{0}\nabla^{3}U([\hat{x} + s(\overline{x}-\hat{x}),\hat{x} + s(\overline{x} - \hat{x}) - (1-s)\frac{h}{2}\sqrt{1-\eta^{2}}\xi])(1-s)\xi ds.
\end{align*}
Therefore we have \eqref{eq:noise_terms} is of the form $\sqrt{1-\eta^{2}}h^{2}A(x,x')\xi + r_{h}$, where 
\[
A(x,x') = \frac{1}{8}\left(2\int^{1}_{0}\nabla^{2} U(\hat{x} + s(\overline{x} - \hat{x}) - (1-s)\frac{h}{2}\sqrt{1-\eta^{2}}\xi)ds - \nabla^{2}U(\overline{x})\right)
\]
is independent of $\xi$,
$
\|A(x,x')\xi\|_{L^{2}} \leq \frac{3}{8}M\sqrt{d},
$
and 
$ \|r_{h}\|_{L^{2}} \leq h^{3}M_{1}d.$
Combining all the estimates we have
\begin{align*}
    \Delta_{v} &= \left(\eta I_{d} - \frac{h^{2}(1+\eta)}{4}Q_{2}\right)(v-v') +\left(-\frac{h\eta}{2}Q_{1} - \frac{h}{2}Q_{2} + \frac{h^{3}(1+\eta)}{8}Q_{2}Q_{1}\right)(x-x') \\
    &+ \epsilon_{v} + h^{2}\sqrt{1-\eta^{2}}A(x,x')\xi,
\end{align*}
where $\|\epsilon_{v}\|_{L^{2}} \leq 2h^{3}\sqrt{d}(M^{3/2}+M_{1}\sqrt{d})$ under Assumptions \ref{assum:G_Lipschitz}-\ref{assum:H_Lipschitz} and $\|\epsilon_{v}\|_{L^{2}} \leq 2h^{2}M\sqrt{d}$ with $A(x,x') = 0$ under Assumptions \ref{assum:G_Lipschitz}-\ref{assum:convex}.
\end{proof}


We will use the following proposition to control the evolution of the error between BAOAB and HOH steps.
\begin{proposition}\label{prop:kernel_estimate}
Consider an HOH scheme, $(x_{i},v_{i})_{i\in \mathbb{N}}$ and a BAOAB scheme $(x'_{i},v'_{i})_{i\in \mathbb{N}}$ initialized at $(x_{0},v_{0}) = (x'_{0},v'_{0}) = (x,v)\sim \pi$ in $\mathbb{R}^{2d}$ with synchronously coupled Gaussian increments and stepsize $h<\frac{1-\eta}{2\sqrt{M}}$, for $l \in \mathbb{N}$ we define $(\Delta^{l}_{x},\Delta^{l}_{v}) :=  (x_{l}-x'_{l},v_{l}-v'_{l})$. For $a = 1/M$ and $b = h/(1-\eta)$ we have that under Assumptions \ref{assum:G_Lipschitz}-\ref{assum:convex}
\[
\|(\Delta^{l}_{x},\Delta^{l}_{v})\|_{L^{2},a,b} \leq \sqrt{3}e^{(l-1)5h\sqrt{M}}C_{l},
\]
where
$
C_{l} = \frac{3lh^{3}M\sqrt{d}}{8}+\frac{2h^{2}\sqrt{Md}}{1-\eta}.
$
Additionally, if Assumption \ref{assum:H_Lipschitz} is satisfied we have
\[
\|(\Delta^{l}_{x},\Delta^{l}_{v})\|_{L^{2},a,b}  \leq \sqrt{3}e^{(l-1)4h\sqrt{M}}C_{l},
\]
where
$
C_{l} := 3h^{3}\sqrt{d}\left(M + \frac{M_{1}}{\sqrt{M}}\sqrt{d}\right)l + \frac{3h^{2}\sqrt{Md}}{8}.
$
\end{proposition}
\begin{proof}[Proof of Proposition \ref{prop:kernel_estimate}]
Let $(x'_{0},v'_{0}) = (x_{0},v_{0}) \sim \pi$ and $(x_{i},v_{i})^{l}_{i=1}$ be defined by the HOH scheme with stepsize $h$ and friction parameter $\gamma$ and $(x'_{i},v'_{i})^{l}_{i=1}$ be defined by the BAOAB scheme with the same stepsize and friction parameter, we define these such that they have synchronously coupled Gaussian increments in the O steps. Then let us define $\Delta^{i}_{x} := x_{i}-x'_{i} $ and $\Delta^{i}_{v} := v_{i}-v'_{i}$ for $i \in \mathbb{N}$, we have by Proposition \ref{prop:BAOAB_local_error} that
\begin{align*}
   \|\Delta^{l}_{x}\|_{L^{2}} &\leq \|\Delta^{l-1}_{x}\|_{L^{2}} + h\left(2hM\|\Delta^{l-1}_{x}\|_{L^{2}} + \|\Delta^{l-1}_{v}\|_{L^{2}}\right) + \frac{3h^{3}M\sqrt{d}}{8}\\
   &\leq \sum^{l-1}_{i=1}h\left(2hM\|\Delta^{i}_{x}\|_{L^{2}} + \|\Delta^{i}_{v}\|_{L^{2}}\right) + l\frac{3h^{3}M\sqrt{d}}{8}.
\end{align*}
Without the additional Assumption \ref{assum:H_Lipschitz} we have
\begin{align*}
   \|\Delta^{l}_{v}\|_{L^{2}} &\leq (\eta + \frac{h^{2}M}{2})\|\Delta^{l-1}_{v}\|_{L^{2}} + hM\left(2 + \frac{h^{2}M}{4}\right)\|\Delta^{l-1}_{x}\|_{L^{2}}  + 2h^{2}M\sqrt{d}\\
   &\leq \sum^{l-1}_{i=1}\eta^{l-1-i}\frac{h}{2}\left(5M\|\Delta^{i}_{x}\|_{L^{2}} + hM\|\Delta^{i}_{v}\|_{L^{2}}\right) + \sum^{l}_{i=1}\eta^{l-i}2h^{2}M\sqrt{d},
\end{align*}
and therefore we have
$
\|(\Delta^{l}_{x},\Delta^{l}_{v})\|_{L^{2},a,0} \leq \sum^{l-1}_{i=1}5h\sqrt{M}\|(\Delta^{i}_{x},\Delta^{i}_{v})\|_{L^{2},a,0} + C_{l},
$
where
$
C_{l} = \frac{3lh^{3}M\sqrt{d}}{8}+\frac{2h^{2}\sqrt{Md}}{1-\eta}
$
and hence
\[
\|(\Delta^{l}_{x},\Delta^{l}_{v})\|_{L^{2},a,0}\leq e^{(l-1)5h\sqrt{M}}C_{l}.
\]
With Assumption \ref{assum:H_Lipschitz} we have
\begin{align*}
    \Delta^{l}_{v} &= \sum^{l-1}_{i=1}\eta^{l-1-i}\left(A^{i}_{v}\Delta^{i}_{v} + A^{i}_{x}\Delta^{i}_{x}\right)+  \sum^{l}_{i=1}\eta^{l-i}\epsilon^{i}_{v} + \sum^{l}_{i=1}\eta^{l-i}h^{2}\sqrt{1-\eta^{2}}A_{i}(x_{i},x'_{i})\xi_{i},
\end{align*}
where $\xi_{i}$ is the noise increment from the iteration $i$ of BAOAB, $A_{i}$ and $\epsilon^{i}_{v}$ are defined by Proposition \ref{prop:BAOAB_local_error} for each iteration and
$
   A^{i}_{x} := -\frac{h\eta}{2}Q^{i}_{1} - \frac{h}{2}Q^{i}_{2} + \frac{h^{3}(1+\eta)}{8}Q^{i}_{2}Q^{i}_{1},
$ $   
   A^{i}_{v} := - \frac{h^{2}(1+\eta)}{4}Q^{i}_{2},
$
where $Q^{i}_{1}$ and $Q^{i}_{2}$ are defined by Proposition \ref{prop:BAOAB_local_error} for each $i$.
Therefore
\begin{align*}
    &\|\Delta^{l}_{v}\|_{L^{2}} \leq \sum^{l-1}_{i=1}\frac{h^{2}M}{2} \|\Delta^{i}_{v}\|_{L^{2}} + 2hM \|\Delta^{i}_{x}\|_{L^{2}} + 2h^{3}\sqrt{d}\left(M^{3/2} + M_{1}\sqrt{d}\right)l + h^{2}\frac{3M\sqrt{d}}{8},
\end{align*}
where we only lose an order of $1/2$ in the last term due to the independence of the Gaussian increments, more precisely for $i \neq j$ and without loss of generality assume $i > j$ we have
$
   \mathbb{E}\left\langle A_{i}\xi_{i},A_{j}\xi_{j}\right\rangle = \mathbb{E}_{\xi_{j}}\mathbb{E}\left[\left\langle A_{i}\xi_{i},A_{j}\xi_{j}\right\rangle \mid \xi_{j}\right] = 0.
$
We therefore have that
\begin{align*}
    \|(\Delta^{l}_{x},\Delta^{l}_{v})\|_{L^{2},a,0} &\leq \sum^{l-1}_{i=1}4h\sqrt{M}\|(\Delta^{i}_{x},\Delta^{i}_{v})\|_{L^{2},a,0} + C_{l} \leq e^{(l-1)4h\sqrt{M}}C_{l},
\end{align*}
where
$
C_{l} := 3h^{3}\sqrt{d}\left(M + \frac{M_{1}}{\sqrt{M}}\sqrt{d}\right)l + \frac{3h^{2}\sqrt{Md}}{8}.
$
\end{proof}

The previous error estimate can be refined using our contraction results, hence we can avoid blowing up exponentially in $l$.
\begin{proof}[Proof of Proposition \ref{prop:second_kernel_estimate}]
We have $(\Delta^{l}_{x},\Delta^{l}_{v}) :=  (x_{l}-x'_{l},v_{l}-v'_{l})$, where $(x_{l},v_{l})$ corresponds to $l$ steps of HOH initiated in $(x_0,v_0)=(x,v)\sim \pi$, and $(x'_{l},v'_{l})$ corresponds to $l$ steps of BAOAB initiated in $(x_0',v_0')=(x,v)$,  driven by the same noise. 

We define a sequence of interpolating variants $(x^{(k)}_l,v^{(k)}_l)$ for every $k=0,1,\ldots, l$ as follows. First, $(x^{(k)}_0 ,v^{(k)}_0)=(x,v)$. We then  define $(x^{(k)}_i,v^{(k)}_i)_{i=1}^{k}$ as HOH steps, followed by $(x^{(k)}_i,v^{(k)}_i)_{i=k+1}^{l}$ as BAOAB steps. So we take $k$ HOH steps, followed by $l-k$ BAOAB steps. From the definition, it follows that
$(x^{(l)}_l,v^{(l)}_l)=(x_l,v_l)$ and 
$(x'^{(0)}_l,v'^{(0)}_l)=(x'_l,v'_l)$. We break $l$ steps into blocks of size $\tilde{l}=\left\lceil \frac{1}{2 \sqrt{M}h}\right\rceil$, as follows,
\begin{align*}
\|(\Delta^{l}_{x},\Delta^{l}_{v})\|_{L^{2},a,b} &= \left\|\left(x^{(0)}_{l}-x^{(l)}_{l},v^{(0)}_{l}-v^{(l)}_{l}\right)\right\|_{L^{2},a,b}\\
&\le \sum_{j=0}^{\lfloor l/\tilde{l} \rfloor -1}  
\left\|\left(x_{l}^{(j \tilde{l})}-x_{l}^{((j+1) \tilde{l})},v_{l}^{(j \tilde{l})}-v_{l}^{((j+1) \tilde{l})}\right)\right\|_{L^{2},a,b}\\
&+\left\|\left(x_{l}^{(\lfloor l/\tilde{l} \rfloor\tilde{l})}-x_{l}^{(l)},v_{l}^{(\lfloor l/\tilde{l} \rfloor \tilde{l})}-v_{l}^{(l)}\right)\right\|_{L^{2},a,b}.
\end{align*}
When we bound the terms $\left\|\left(x_{l}^{(j \tilde{l})}-x_{l}^{((j+1) \tilde{l})},v_{l}^{(j \tilde{l})}-v_{l}^{((j+1) \tilde{l})}\right)\right\|_{L^{2},a,b}$, we can use the fact that the first $j\tilde{l}$ according to HOH keep the stationary distribution invariant, and the two chains deviate in the following $\tilde{l}$ steps (since one of them is doing HOH, and the other one is doing BAOAB steps). Still, the remaining steps are doing BAOAB for both chains (hence, there is a contraction).
Using Proposition \ref{prop:kernel_estimate} with $l$ chosen as $\tilde{l}$, and Theorem \ref{Theorem:BAOAB}, we have
\begin{align*}
\left\|\left(x_{l}^{(j\tilde{l})}-x_{l}^{((j+1)\tilde{l})},v_{l}^{(j \tilde{l})}-v_{l}^{((j+1) \tilde{l})}\right)\right\|_{L^{2},a,b} \le \sqrt{3}e^{5/2} C_{\tilde{l}} \cdot 7(1-c(h))^{\frac{l-1-(j+1) \tilde{l}}{2}}.
\end{align*}
Under Assumptions \ref{assum:G_Lipschitz}-\ref{assum:convex}, $C_{\tilde{l}}=\frac{3\tilde{l}h^{3}M\sqrt{d}}{8}+\frac{2h^{2}\sqrt{Md}}{1-\eta}\le \frac{3h^{2}\sqrt{Md}}{1-\eta}$. If we also include Assumption \ref{assum:H_Lipschitz},
$C_{\tilde{l}} := 3h^{3}\sqrt{d}\left(M + \frac{M_{1}}{\sqrt{M}}\sqrt{d}\right)\tilde{l} + \frac{3h^{2}\sqrt{Md}}{8}\le h^2(4\sqrt{Md}+3\frac{M_1}{M} d)$. By some simple algebra, we have that
\begin{align*}
&\|\Delta^{l}_{x},\Delta^{l}_{v}\|_{L^{2},a,b}\le \sqrt{3} e^{5/2} C_{\tilde{l}}\left(1+ \frac{7}{1-(1-c(h))^{\tilde{l}/2}}\right)\le 170 C_{\tilde{l}}\cdot \frac{1}{1-e^{-c(h)\tilde{l}/2}}\\
&=  170 C_{\tilde{l}}\cdot \frac{1}{1-e^{-\frac{h^{2}m}{4\left(1 - \eta\right)}\tilde{l}/2}}\le 
170 C_{\tilde{l}}\cdot \frac{1}{1-e^{-\frac{h m}{8\sqrt{M}\left(1 - \eta\right)}}}\le 170 C_{\tilde{l}}\cdot \left(1+\frac{8\sqrt{M}\left(1 - \eta\right)}{h m}\right),
\end{align*}
where we have used the fact that $1/(1-e^{-x})\le 1+1/x$ for $x>0$. The claims now follow by rearrangement.
\end{proof}

\end{document}